\documentclass[11pt]{amsart}
\usepackage{graphicx, amsmath, fullpage, amssymb, amsthm, amsfonts, color, algorithm}
\usepackage{enumerate}
\newtheorem{theorem}{Theorem}[section]
\newtheorem{lemma}[theorem]{Lemma}
\newtheorem{prop}[theorem]{Proposition}
\newtheorem{cor}[theorem]{Corollary}

\theoremstyle{definition}
\newtheorem{definition}[theorem]{Definition}

\theoremstyle{remark}
\newtheorem{remark}[theorem]{Remark}
\numberwithin{equation}{section}
\usepackage[colorlinks,
            linkcolor=red,
            anchorcolor=red,
            citecolor=blue
            ]{hyperref}
 
\newcommand{\CNBW}{\textnormal{CNBW}}

\renewcommand{\epsilon}{\varepsilon}

\begin{document}

\title{Global eigenvalue fluctuations  of random biregular bipartite graphs}

\author{Ioana Dumitriu}

\address{Department of Mathematics, University of California, San Diego, La Jolla, CA 92093}

\email{idumitriu@ucsd.edu}

\author{Yizhe Zhu}
\address{Department of Mathematics, University of California Irvine, Irvine, CA 92697}
\email{yizhe.zhu@uci.edu}
\thanks{This work was partially supported by NSF DMS-1949617}

\date{\today}

\maketitle

\begin{abstract}
We compute the eigenvalue fluctuations of uniformly distributed random biregular bipartite graphs with fixed and growing degrees for a large class of analytic functions. As a key step in the proof, we  obtain a total variation distance bound for the Poisson approximation of the number of cycles and cyclically non-backtracking walks in random biregular bipartite graphs, which might be of independent interest. We also prove a semicircle law for random $(d_1,d_2)$-biregular bipartite graphs when $\frac{d_1}{d_2}\to\infty$.
As an application, we translate the results to adjacency matrices of uniformly distributed  random regular hypergraphs.
\end{abstract}

\section{Introduction}
\subsection{Eigenvalue fluctuations of random matrices}
The study of fluctuations from the limiting empirical spectral distributions (ESDs) for random matrices is a well-established topic of interest in random matrix theory, originated in \cite{jonsson1982some,khorunzhy1996asymptotic, sinai1998central}, see also \cite{anderson2010introduction} and all references therein. More recently, it has been extended to sparse random matrices and random graph-related matrices in various regimes of sparsity and independence (\cite{shcherbina2010central,shcherbina2012central,benaych2014central,dumitriu2013functional,ben2015fluctuations}), and the natural next target is hypergraphs (\cite{feng1996spectra}). 

The ultimate goal in these studies is to see the equivalent of the one-dimensional Central Limit Theorem (CLT) emerge, when examining linear statistics of the spectra of random matrices and random graphs. More precisely, denote by $\lambda_1, \ldots, \lambda_n$ the eigenvalues of the random matrix, suitably scaled to put them with high probability on a compact set, and let $f$ be a suitably smooth function.  When the matrices in question are not extremely sparse, one can almost invariably prove that the linear statistic
\[
\mathcal{L}(f) = \sum_{i=1}^n f(\lambda_i) 
\]
has the property that, when centered, it converges to a normal distribution whose variance depends on $f$:
\[
\mathcal{L}(f) - \mathbb{E} (\mathcal{L}(f)) ~\rightarrow ~ N(0, \sigma_f^2)~.
\]

\subsubsection{Dense and not-too-sparse Wigner cases.}
There is an interesting phenomenon taking place with respect to sparsity; the variance $\sigma_f^2$ is the same in the case of Gaussian Orthogonal Ensembles (GOEs) as in the case of the random regular graph under the  permutation model with growing degrees \cite{dumitriu2013functional}:
\begin{eqnarray} \label{varf}
\sigma_f^2 & = & 2 \sum_{k=1}^{\infty} ka_k^2~,
\end{eqnarray}
where $a_k$ is the $k$-th coefficient of $f$ in the Chebyshev polynomial basis expansion. Small variations of this expression also occur in dense Wigner variants and the uniform regular graph model, as follows. For real  Wigner and generalized Wigner matrices in the dense case \cite{johansson1998fluctuations,bai2005convergence,anderson2006clt,chatterjee2009fluctuations,sosoe2013regularity}, $\sigma_f$ also depends on the fourth moments of the off-diagonal entries and the variance of the diagonal entries, which yields corrections to the constants in front of $a_1^2, a_2^2$ (see Theorem 1.1 in \cite{bai2005convergence} for an explicit expression). Similarly, in the uniform regular graph model \cite{johnson2015exchangeable}, a correction must be introduced as there are no one- or two-cycles (as the graph is simple), and so the terms corresponding to $k=1$ and $k=2$ in the sum \eqref{varf} are not present.  

However, in the case of sparse Wigner matrices (corresponding to Erd\H{o}s-R\'enyi graphs $G(n,p)$ with $p\to 0, np\to\infty$, \cite{shcherbina2012central}), the fluctuations are impacted by the fact that the number of nonzero entries in each row (i.e., the degree of each vertex) fluctuates, and the $4$th moment of the scaled adjacency matrix entries grows. The variance $\sigma_f^2$ blows up, necessitating another multiplicative scaling of the linear statistic $\mathcal{L}(f) - \mathbb{E}(\mathcal{L}(f))$ by $\sqrt{p}$, and extracting only part of the expression \eqref{varf} (see Theorem 1 in \cite{shcherbina2012central}). 

\subsubsection{Dense Wishart cases.}
A similar phenomenon occurs in the Wishart case, i.e., for sample covariance matrices (corresponding to bipartite graphs); in the case of dense matrices with converging aspect ratio, the variance is given in different forms in \cite{bai2004clt, bai2010functional}. Although these expressions are not explicit in terms of a Chebyshev polynomial expansion, in \cite{cabanal2001fluctuations,kusalik2007orthogonal}, it is shown that the covariance between two linear statistics is diagonalized by shifted Chebyshev polynomials.  When the aspect ratio goes to $\infty$, \cite{chen2015clt} computes the variance which is consistent with the Wigner case in \cite{bai2005convergence}. So far, we are not aware of any CLT results for sparse bipartite Erd\H{o}s-R\'{e}nyi graphs, but a similar argument as in \cite{shcherbina2012central} should apply.

For dependent entries (biregular bipartite graphs), we obtain here the variance of the eigenvalue fluctuation in Theorem \ref{thm:GaussianCLT}, and it matches the one in \cite{chen2015clt}, except for the first coefficient.

\subsubsection{Constant (expected or deterministic) degree.} 
When $p=\frac{c}{n}$, the explicit limiting spectral distribution for Erd\H{o}s-R\'{e}nyi graphs $G(n,p)$ is not known, although it is known that the measure $\mu_c$ exists for every $c$ (given, e.g., by a Stieltjes transform equation as in \cite{bordenave2010resolvent}), and if $c>1$ it consists of a continuous part and an atomic part \cite{bordenave2017mean}. 
Convergence of $\mu_c$ to the semicircular distribution is studied in  \cite{enriquez2016spectra,jung2018delocalization}, where asymptotic expressions for the moments of $\mu_c$ with an $o(1/c)$ term are computed (as $c \rightarrow \infty$, $\mu_c$ converges to the semicircle law).

However, a CLT for Erd\H{o}s-R\'{e}nyi graphs $G(n,\frac{c}{n})$ still holds \cite{shcherbina2010central,benaych2014central} with a more complicated variance that does not follow the same expression as in \eqref{varf}, see Theorem 2.2 in \cite{benaych2014central}. 
By contrast, in the random $d$-regular graph case with $d$ finite, the fluctuations are no longer Gaussian. Instead, they are modeled by an infinitely divisible distribution, expressed as a sum of Poisson variables (see \cite{dumitriu2013functional} for the permutation model, and \cite{johnson2015exchangeable,metz2014finite} for the uniform model). Notably, in the case when the matrix is not symmetric and corresponds to the (directed) cycle structure of a random permutation, \cite{ben2015fluctuations} showed that the global fluctuations could be computed, and whether or not the limiting distribution is Gaussian depends on how smooth the test function is.  For random regular graphs with fixed degree $d$, we will not see the effect of degree fluctuation. But for Erd\H{o}s-R\'{e}nyi graphs $G(n,\frac{d}{n})$ with expected degree $d$, the degree fluctuation contributes to the fluctuation of linear statistics, and there is an extra $n^{-1/2}$ normalization in the CLT (see \cite[Theorem 4]{shcherbina2010central}). Such a difference shows that eigenvalues of random $d$-regular graphs for fixed $d$ are more rigid than the corresponding Erd\H{o}s-R\'{e}nyi graph $G(n,\frac{d}{n})$. 

For the bipartite Erd\H{o}s-R\'enyi case, once again, the limiting distribution is not known, but results that are similar to \cite{enriquez2016spectra} can be found in \cite{noiry2018spectral}.  We are not aware of any CLT-like results for the fluctuations in this case.
We compute the fluctuations for the uniformly random biregular bipartite with fixed degrees. Just like in the regular case \cite{dumitriu2013functional}, we see that the fluctuations are modeled by a sum of Poisson variables (Theorem \ref{thm:CLTfixed}). 

% \begin{remark} In all cases where a formula for the variance has been obtained by expressing $f$ in a polynomial basis which diagonalizes the covariance, the correct basis has been  given by the orthogonal polynomials with respect to the limiting distribution, \cite{sodin2007random,dumitriu2013functional}. 
% \end{remark}

Another important  class of random graphs  is the configuration model. The ESD of the configuration model with a large mean degree is not generally given by the  semicircle law \cite{dembo2021empirical,metz2020spectral}, but no linear statistics result is known. When the degrees grow with the size of the graph, similar to Erd\H{o}s-R\'enyi graphs and random regular graphs, we expect a central limit theorem for linear statistics holds. There are linear statistics results for inhomogeneous matrix models beyond Wigner matrices \cite{chatterjee2009fluctuations,adhikari2021linear}, and it might be possible to apply their techniques, together with a coupling argument introduced in \cite{dembo2021empirical}, to study the linear statistics for the configuration model.

\subsection{Random biregular bipartite graphs}\label{sec:RBBG}
Biregular bipartite graphs have found applications in  error correcting codes, matrix completion, and community detection, see for example \cite{hoory2006expander,tanner1981recursive,sipser1996expander, gamarnik2017matrix,brito2018spectral,burnwal2020deterministic,blake2017short,dehghan2018tanner,dehghan2019computing}.
An \textit{$(n,m,d_1,d_2)$-biregular bipartite  graph} is a bipartite graph $G=(V_1,V_2, E)$ where $|V_1|=n, |V_2|=m$ and every vertex in $V_1$ has degree $d_1$ and every vertex in $V_2$ has degree $d_2$. Here we must have $nd_1=md_2= |E|$. When the number of vertices is clear, we call it a \textit{$(d_1,d_2)$-biregular bipartite graph} for simplicity.

Let $X\in \{0,1\}^{n\times m}$ be a matrix indexed by $V_1\times V_2$ such that 
$X_{ij}=1$ if and only if $ (i,j)\in E$. The adjacency matrix of a $(d_1,d_2)$-biregular bipartite graph with $V_1=[n], V_2=[m]$ can be written as 
\begin{align}\label{eq:A}
    A=\begin{bmatrix}
    0  &X\\
    X^{\top} &0 
    \end{bmatrix}.
\end{align}
All eigenvalues of $A$ come in pairs as $\{-\lambda, \lambda\}$, where $|\lambda|$ is a singular value of $X$, along with extra $|n-m|$ zero eigenvalues. It's easy to see  $\lambda_1(A)=-\lambda_{n+m}(A)=\sqrt{d_1d_2}$. 

The empirical spectral distribution for uniformly distributed random biregular bipartite graphs (RBBGs), which is the equivalent of the Kesten-McKay law, was first computed in \cite{godsil1988walk,mizuno2003semicircle} for the case of fixed $d_1, d_2$. For growing degrees, when $\frac{d_1}{d_2}$ converges to a positive constant, the analogue to the Mar\v{c}enko-Pastur law was proved in \cite{dumitriu2016marvcenko, tran2020local,yang2017local}.

In this paper, instead of examining the spectrum of $A$, we will be looking at the spectrum of the matrix $XX^{\top} - d_1I$. This serves two purposes: one, it allows for an immediate parallel to the sample covariance matrix (Wishart) case, and two, it allows us to deal with all regimes in a unitary fashion. The eigenvalues of $XX^{\top}-d_1I$ are the shifted squares of the eigenvalues of $A$. Any result on global fluctuations for linear statistics of the spectrum of $XX^{\top}-d_1I$ is automatically converted into an equivalent result for the spectrum of $A$. However, because any result of fluctuations must necessarily put most of the eigenvalues (with the exception of the deterministic outliers) on a compact interval, scaling must be involved. This works perfectly fine when the ratio $d_1/d_2$ is bounded, but it becomes tricky when it is not, and the matrix $XX^{\top}-d_1I$ allows us to do the scaling in a more natural way, similarly to the sample covariance (Wishart) matrix with unbounded aspect ratio in \cite{chen2015clt}.

To prove a result on eigenvalue fluctuations, we need two special ingredients: \emph{eigenvalue confinement on a compact interval} and \emph{asymptotic behavior of cycle counts}. For the former, we make use of the spectral gap shown in \cite{brito2018spectral} for the fixed degree case and \cite{zhu2020second} for the growing degree case. Previous results  of this kind were obtained for random regular graphs \cite{friedman2008proof,bordenave2015new} for a fixed degree, and \cite{broder1998optimal,cook2018size,tikhomirov2019spectral,bauerschmidt2020edge} for growing degrees.

For the latter, we use Stein's method to approximate cycle counts as Poisson random variables by bounding the total variation distance (Theorem \ref{thm:Poissoncount}) and obtain a Poisson approximation of the number of cyclically non-backtracking walks (Corollary \ref{cor:dTVCNBW}). 
Note that computing cycle counts is a fundamental problem in the study of random graphs, ever since the seminal papers of \cite{mckay1981expected} and more general \cite{mckay1981subgraphs,mckay2004short}. 

To prove our results, we follow the recipe of \cite{johnson2015exchangeable}  by using switching to construct exchangeable pairs of graphs that allow us to estimate cycle counts. The switching we use here differs from \cite{johnson2015exchangeable} and is suitable for  biregular bipartite graphs. In the analysis of switchings, a new challenge is an imbalance between the parameters $d_1,d_2$ when the aspect ratio is unbounded. Our  results on cycle counts hold for a large range of $d_1,d_2$, and are notably independent of the aspect ratio as long as the cycle length is small. It is also worth noting that the method of switching has been applied to other problems on random  biregular  bipartite graphs, for example,  \cite{canfield2005asymptotic,canfield2008asymptotic,perarnau2013matchings}.

Finally, we also obtain an algebraic relation between linear eigenvalue statistics on modified Chebyshev polynomials and cyclically non-backtracking walks (Theorem \ref{thm:chebyshevCNBW}). Then based on the spectral gap results in \cite{brito2018spectral,zhu2020second} and approximation theory for Chebyshev polynomials \cite{trefethen2013approximation}, we extend the eigenvalue fluctuation results to a general class of analytic functions.

\subsection{Main results} Our main contributions are represented by Theorems \ref{thm:CLTfixed}, \ref{thm:GaussianCLT}, establishing the behavior of the global fluctuations for the linear statistics of eigenvalues of RBBGs in the fixed $d_1, d_2$, respectively, in the $d_1\cdot d_2 \rightarrow \infty$ cases.  Note that Theorem \ref{thm:GaussianCLT} describes the behavior of the fluctuations even in the case when the limiting ESD does not exist since it merely requires $d_1/d_2$ to be bounded, rather than to converge to a number in $[1, \infty)$ (which would be the necessary condition for the ESD to converge). In 
addition, we  show that the covariance between two linear statistics with different test functions is given by the coefficients in their Chebyshev expansions.

As part of the proofs for our main results, we also describe the asymptotic behavior of the cycle counts (Theorem \ref{thm:Poissoncount}). Based on the cycle counts estimates, we then use the locally tree-like structure of RBBGs to prove a global semicircle law in the case when the degree goes slowly \textcolor{blue}{$(d_1=n^{o(1)})$} and  $d_1/d_2$ is unbounded (Theorem \ref{thm:globallaw}).

Finally, as an important application, we obtain  equivalent results for uniformly distributed random regular  hypergraphs, including cycle counts, global laws, spectral gaps, and eigenvalue fluctuations.

\subsection{Organization of the paper}
In Section \ref{sec:cycle} we prove our results on cycle counts in random biregular bipartite graphs. Section \ref{sec:spectralgap} collects relevant results for the spectral gap and eigenvalue confinement on a compact interval from the literature. Section \ref{sec:CLT} proves our main results, Theorems \ref{thm:CLTfixed} and \ref{thm:GaussianCLT}. Section \ref{sec:globallaw} proves a global semicircle law for RBBGs when $d_1/d_2$ is unbounded. In  Section \ref{sec:hypergraph}, we use the connections established in \cite{dumitriu2019spectra} to prove several results on uniformly distributed regular hypergraphs.

\section{Cycle counts}\label{sec:cycle}
\subsection{Counting switchings}
In this section, we estimate the number of switchings that create or delete a cycle in a biregular bipartite graph. The precise definitions of switchings for our purposes are given in Definition \ref{def:fswitch} and Definition \ref{def:bswitch}.  These estimates will be used in Section \ref{sec:Poi} to show that cycle counts converge in distribution to Poisson random variables.

\begin{definition}[cycle]\label{def:simplecycle}
Throughout the paper, when we say a  \textit{cycle}, we mean a \textit{simple cycle}, i.e., all vertices in a cycle are distinct.
\end{definition}

Let $K_{n,m}$ be the complete bipartite graph on $n+m$ vertices with $V_1=[n], V_2=[m]$.
Let $H\subseteq K_{n,m}$ be a subgraph with $v$ vertices. For any $i\in K_{n,m}$, let $g_i,h_i$ denote the degree of $i$ considered as a vertex in a biregular bipartite graph $G=(V_1,V_2, E)$ and the subgraph $H$, respectively. Let $h_{\max}$ be the largest value of $h_i$ and  $|H|$ be the number of edges of $H$. Denote by 
$$[x]_a=x(x-1)\cdots(x-a+1)$$ the falling factorial.  The following estimate is given in   \cite{mckay1981subgraphs}.
\begin{prop}[Theorem 3.5 in \cite{mckay1981subgraphs}] \label{eq:propH}
	Assume $d_1\geq d_2$ and $nd_1\geq 2d_1(d_1+h_{\max}-2)+|H|+1$. Then
	\[ \mathbb P(H\subseteq G)\leq \frac{\prod_{i=1}^v [g_i]_{h_i}}{[nd_1-4d_1^2-1]_{|H|}}.\]
\end{prop}
We first prove several estimates on random biregular bipartite graphs based on Proposition \ref{eq:propH}.

\begin{lemma}\label{lem:Mckay}
Let $G$ be a random $(d_1,d_2)$-biregular bipartite graph with $d_2\leq d_1\leq n^{1/3}$.
\begin{enumerate}
	\item Suppose $H$ is a subgraph of the complete graph $K_{n,m}$ in which every vertex has degree at least $2$. Let $e$ be the number of edges  in $H$. Suppose $e=o(n^{1/3})$. Then
	\begin{align} \label{eq:Mckay0}
	\mathbb P(H\subseteq G)\leq c_1\left(\frac{(d_1-1)(d_2-1)}{nm}\right)^{e/2}.
	\end{align}
\item Let $\alpha$ be a cycle of length $2k$ in the complete bipartite graph $K_{n,m}$. Suppose   $k\leq n^{1/10}$, then
 \begin{align}\label{eq:McKay1}
     \mathbb P(\alpha\subseteq G)\leq c_1\left(\frac{(d_1-1)(d_2-1)}{nm}\right)^k.
 \end{align} 
 \item Let $\beta$ be another cycle of length $2j\leq 2n^{1/10}$ in the complete bipartite graph $K_{n,m}$. Suppose $\alpha,\beta$ share $f$ edges. Then
 \begin{align}\label{eq:McKay2}
 \mathbb P(\alpha\cup \beta \subseteq G)\leq c_1\left(\frac{(d_1-1)(d_2-1)}{nm}\right)^{j+k-f/2}. 
 \end{align} 
\end{enumerate}
\end{lemma}

\begin{proof}It suffices to prove \eqref{eq:Mckay0}. Then \eqref{eq:McKay1} and \eqref{eq:McKay2} follow as special cases. Since $H$ has $e$ edges, and $H$ is bipartite, it satisfies \[ \sum_{i\in V_1}h_i=\sum_{i\in V_2}h_i=e.\]
Since $h_i\geq 2$ for all $i\in V(H)$, we know $[g_i]_{h_i}\leq (g_i(g_i-1))^{h_i/2}$. Therefore from Proposition \ref{eq:propH},
\begin{align*}
\mathbb P(H\subseteq G)&\leq \frac{(d_1(d_1-1))^{e/2}(d_2(d_2-1))^{e/2}}{[nd_1-4d_1^2-1]_{e}}\\
&=\left(\frac{(d_1-1)(d_2-1)}{nm}\right)^{e/2} \frac{(nd_1)^{e}}{[nd_1-4d_1^2-1]_{e}}.
\end{align*}

Recall $d_1\leq n^{1/3},e=o(n^{1/3})$, and $(1+x)^r=1+O(rx)$ if $rx\to 0$. We have for some absolute constant $c_1>0$,
\begin{align}\label{eq:approx}
\frac{(nd_1)^{e}}{[nd_1-4d_1^2-1]_{e}}\leq \left(\frac{nd_1}{nd_1-4d_1^2-e }\right)^{e}=\left( 1+\frac{4d_1^2+e}{nd_1-4d_1^2-e}\right)^{e}\leq c_1.
\end{align}
 This proves \eqref{eq:Mckay0}.
\end{proof}

 Let $G$ be a $(d_1,d_2)$-biregular bipartite graph.  Let $C_j$ be the number of cycles of length $2j$ in $G$. We will always represent a cycle by a vertex sequence starting from a vertex in $V_1$.  Suppose $\alpha=(x_1,y_1,\cdots, x_{k},y_{k})$ is a cycle of length $2k$ in $G$ with $x_{i}\in V_1, y_{i}\in V_2$, $1\leq i\leq k,$ where   $y_{k}$ is connected to $x_1$ in the cycle $\alpha$.

  Let $e_i=u_iv_i, e_i'=u_i'v_i'$  be the  edges with  with $u_i, u_i'\in V_1,v_i, v_i'\in V_2$, $1\leq i\leq k$  such that neither $u_{i},u_{i}'$ is adjacent to $y_{i}$  for $1\leq i\leq k$ and neither $v_{i}, v_{i}'$ is adjacent to $x_i$. See the left part of Figure \ref{fig:switching} for an example.

 We now introduce our definitions of switching for biregular bipartite graphs.
 \begin{definition}[forward $\alpha$-switching]\label{def:fswitch}
  Consider the action of deleting all  $4k$  edges  $\tilde{e}_i, 1\leq i\leq 2k$ and  $e_i,e_i', 1\leq i\leq k$, and replacing them by the edges $x_iv_{i}, x_iv_{i}', y_iu_{i}, y_iu_{i}'$ for $1\leq i\leq k$.   We obtain a new biregular bipartite graph $G'$ with the cycle $\alpha$ deleted. We call this action induced by the 6 sequences $(x_i),(y_i), (u_i),(u_i'), (v_i),(v_i'),  1\leq i\leq k$ a \textit{forward $\alpha$-switching}.
  See Figure \ref{fig:switching} for an example. We will consider forward $\alpha$-switchings
only up to cyclic rotation and inversion of indices in $[k]$; that is, we identify the $2k$ different forward $\alpha$-switchings obtained by applying the same cyclic rotation or inversion on $[k]$ to the 6 sequences $(x_i),(y_i), (u_i),(u_i'), (v_i),(v_i'), 1\leq i\leq k$. 
 \end{definition}
 
  \begin{figure}[ht]
 \includegraphics[width=12cm]{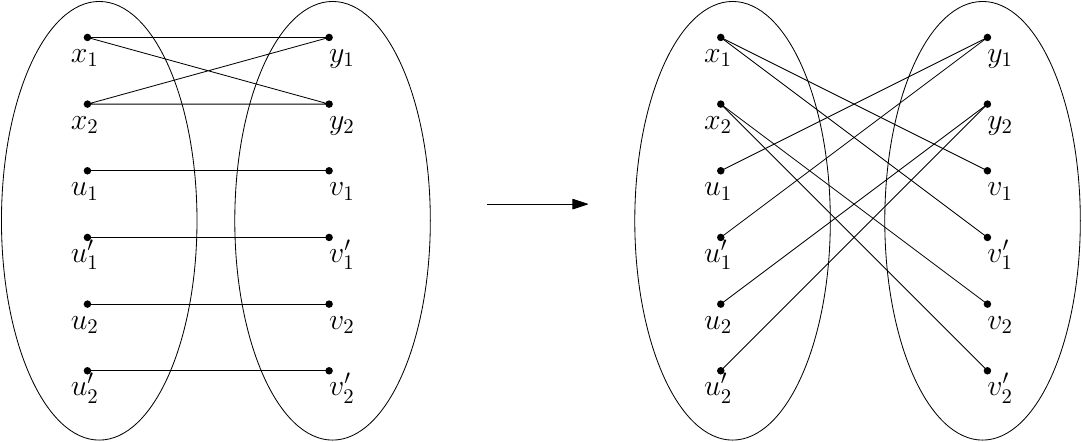}
 \caption{A forward $\alpha$ switching from the left to the right, where $\alpha=(x_1,y_1,x_2,y_2)$.}	\label{fig:switching}
 \end{figure}

\begin{definition}[backward $\alpha$-switching]\label{def:bswitch}
Suppose $G$ contains  paths $v_{i}x_iv_{i}'$ and $u_{i}y_iu_{i}'$ for $1\leq i\leq k$, where $x_i,u_i, u_i'\in V_1,  y_i,v_i,v_i'\in V_2$. Consider  deleting all $4k$  edges $v_{i}x_i, v_i'x_{i}, u_{i}y_i, u_i'y_i$  for $1\leq i\leq k$,  and replacing them with $u_iv_i, u_i'v_i'$, $x_iy_i, y_ix_{i+1}$ for $1\leq i\leq k$. We obtain a new graph $G'$ with a cycle $\alpha=(x_1,y_1,\cdots, x_{k},y_{k})$. Such action is called a \textit{backward $\alpha$-switching} induced by the sequences $(x_i),(y_i), (u_i),(u_i'), (v_i),(v_i'), 1\leq i\leq k$. We also identify the $2k$ different backward $\alpha$-switchings obtained by applying the same cyclic rotation or inversion on the index set $[k]$.
\end{definition}

\begin{definition}[short cycles] Let $r$ be an integer; we say that a cycle is \textit{short} if its length is less than or equal to $2r$. \end{definition}

 We call  a $\alpha$-switching \textit{valid} if $\alpha$ is the only short cycle created or destroyed by the switching. For each forward $\alpha$-switching from $G$ to $G'$, there is a corresponding backward $\alpha$-switching from $G'$ to $G$ by simply reversing the operation (i.e. from right to left  in Figure \ref{fig:switching}).

 Let $F_{\alpha}$  be the number of all valid forward $\alpha$-switchings from $G$ to some $G'$ and let $B_{\alpha}$ be the number of all valid backward $\alpha$-switchings from some $G'$ to $G$. In the following two lemmas, we estimate $F_{\alpha}$ and  $B_{\alpha}$ for biregular bipartite graphs.

 \begin{lemma}\label{lem:F}
Let $G$ be a deterministic $(d_1,d_2)$-biregular bipartite graph with $d_1\geq d_2$   and  cycle counts $C_k, 2\leq k\leq r$. For any short cycle $\alpha\subseteq G$ of length $2k$, we have
\begin{align}\label{eq:upperF}
F_{\alpha}\leq [n]_{k}[m]_kd_1^{k}d_2^k.	
\end{align}
If $\alpha$ does not share an edge with another short cycle, then for an absolute constant $c_1>0$, we have
\begin{align}\label{eq:lowerF}
F_{\alpha}\geq [n]_{k}[m]_k	d_1^{k}d_2^k \left( 1-\frac{4k\sum_{j=2}^rj C_j+c_1k(d_1-1)^r(d_2-1)^r}{nd_1}\right).
\end{align}	
 \end{lemma}

\begin{proof}

Consider a cycle denoted by $\alpha=(x_1,y_1,\cdots, x_{k},y_{k})$. Denote  edges 
 \begin{align}\label{eq:etilde}
    \tilde{e}_{i}=x_iy_{i},  \quad   \tilde{e}_{i+k}=y_ix_{i+1}, 1\leq i\leq k,  
 \end{align} 
 where $x_{k+1}:=x_1$. There are at most $[n]_{k}d_1^{k}[m]_kd_2^k$
many ways to choose  edges $e_i=u_iv_i$ and $e_i'=u_i'v_i'$ for $1\leq i\leq k$, which gives the upper bound \eqref{eq:upperF}.  For the  $k$ edges $e_i, 1\leq i\leq k$, we require distinct $u_i\in V_1, 1\leq i\leq k$, and we have $d_1$ choices for each $v_i$, given the degree constraint on $u_i$. This gives $[n]_kd_1^k$ many choices altogether. For the remaining  edges $e_i', 1\leq i\leq k$ we require distinct $v_i'\in V_2, 1\leq i\leq k$ and each for each $u_i'$ we have $d_2$ choices, giving us a factor of $[m]_k d_2^k$.  Therefore \eqref{eq:upperF} holds.

 For the rest of the proof,  we always use the same way to count $\alpha$-switchings by counting the choices of  $e_i,e_i'$. We use the parameter $d_1$ to control the choices from $e_i, 1\leq i\leq k$ and the parameter $d_2$ for the choices from $e_i',1\leq i\leq k$.

To prove the lower bound in \eqref{eq:lowerF}, we choose a subset of configurations that are guaranteed to have a valid forward $\alpha$-switching. Consider $e_i, e_i', 1\leq i\leq k$ such that the following holds: 
\begin{enumerate}
	\item $e_i$ and $e_i'$ are not contained in any short cycle in $G$ for $1\leq i\leq k$.\label{cond1}
\item The distance from any vertex in  $\{e_i,e_i'\}$  to any vertex in $\tilde{e}_i$ is at least $2r$ for any $1\leq i\leq k$.\label{cond2}
	\item The distance between any two different edges among the $2k$ edges $\{e_i,e_i', 1\leq i\leq k\}$ is at least $r$. \label{cond3}
	\item  For all $1\leq i\leq k$,
	 the distance between  $v_{i}$ and $v_{i}'$ is at least $2r$, and the distance between  $u_{i}$ and $u_{i}'$ is at least $2r$.\label{cond4}
	 \end{enumerate}

Recall the definition of $\tilde{e}_i$ in \eqref{eq:etilde}. By Condition \eqref{cond2}, for all $1\leq  i\leq k$, $u_{i},u_{i}'$ are not adjacent to $y_i$, also $v_{i},v_{i}'$ are not adjacent to $x_i$, which satisfies the definition of a forward $\alpha$-switching. Let $G'$ be the  graph obtained  by applying the forward $\alpha$-switching from $G$. We need to check that $\alpha$ is the only cycle deleted in $G$ by this switching, and no other short cycles are created in $G'$. 

Since $\alpha$ shares no edges with other short cycles by our assumptions, deleting $\alpha$ will not destroy other short cycles. From Condition \eqref{cond1}, deleting $e_i,e_i'$ will not destroy any short cycles either.  

Next, we show no other short cycles are created in $G'$. Suppose there exists a new short cycle $\beta$ in $G'$ created by the switching. Then $\beta$ contains paths in $G\cap G'$ separated by edges created in the forward switching   in $G'$ ($\beta$ must contain at least such edge because it is created). Any such path in $G\cap G'$ must have a length at least $r$, because
\begin{itemize}
	\item if it starts and ends at vertices in  $\alpha$ and has length less than $r$, then combining this path with a path in $\alpha$ gives a short cycle in $G$ that intersects $\alpha$, which is a contradiction to our assumption on $\alpha$;
	\item if it starts in $\alpha$ and ends in $\{u_i,v_i, u_i', v_i'\}$ for some $i$ and has length less than $r$, then combining this path with a path in $\alpha$ gives a path between $\tilde{e_i}$, $e_i$ or between $\tilde{e_i}$, $e_i'$ of length less than $2r$, which violates Condition \eqref{cond2};
	\item if it starts and ends in different edges among $\{e_i, e_i',1\leq i\leq k\}$, then it must have length at least $r$ by Condition \eqref{cond3}; 
	\item if it starts at some vertex in $e_i$ and ends at some vertex in $e_i$, then the path must start and end at different vertices in $e_i$. Otherwise,  $\beta$ is not a cycle in the sense of Definition \ref{def:simplecycle}. Then the path combined with $e_i$ is a cycle.  By Condition \eqref{cond1},  it has a length at least $r$, a contradiction. In the same way, it cannot start at some vertex in $e_i'$ and end at some vertex in $e_i'$.
\end{itemize}
This implies $\beta$ contains exactly one path in $G\cap G'$. If not, the two separated paths together with new edges in $G'$ have lengths greater than $2r$, a contradiction to the condition that $\beta$ is a short cycle.
Given the path in $G\cap G'$, the remainder of $\beta$ has two cases: 
\begin{itemize}
	\item a single edge that can be $x_iv_{i}$, $x_iv_i'$, $y_iu_{i}$, or $y_iu_i'$  for some $1\leq i\leq k$, then by Condition \eqref{cond2}, the path in $G\cap G'$ connecting the two vertices in the edge has length at least $2r$, which is a contradiction to the fact that $\beta$ is a short cycle; 
	\item a single path $v_{i}x_iv_{i}'$ or $u_{i}y_iu_{i}'$, which is impossible by Condition \eqref{cond4}.
\end{itemize}

From the analysis above, no such $\beta$ can exist, hence any $\alpha$-switching satisfying Conditions \eqref{cond1}-\eqref{cond4} is valid. 

Next, we find the number of all switchings  satisfying Conditions \eqref{cond1} to \eqref{cond4} to have a lower bound on $F_{\alpha}$. We will do this by bounding from above the number of switchings out of the $[n]_{k}[m]_kd_1^{k}d_2^k$ many choices counted in \eqref{eq:upperF} that fail one of the Conditions \eqref{cond1}-\eqref{cond4}. We treat the $4$ conditions in the following (a)-(d) parts. 

 (a) There are a total of at most $\sum_{i=2}^r 2jC_j$ edges in all short cycles of $G$.  For some $1\leq i\leq k$, if we choose  one edge $e_i$  from a short cycle  and the other $(2k-1)$    edges arbitrarily,  we obtain a forward $\alpha$-switching that fails Condition \eqref{cond1}. The number of all possible choices is at most
 \[k\sum_{j=2}^{r}2jC_j \cdot 	[n-1]_{k-1}[m]_kd_1^{k-1}d_2^k.\]
 And if we choose $e_i'$ from a short cycle and the other $(2k-1)$  edges arbitrarily, the number of all possible choices is at most
 \[k\sum_{j=2}^{r}2jC_j \cdot 	[n]_{k}[m-1]_{k-1}d_1^{k}d_2^{k-1}.\]
 Altogether the number of  choices is at most 
	\begin{align}\label{eq:abcd1}
 \frac{4}{nd_1}[n]_{k}[m]_k(d_1d_2)^kk\sum_{j=2}^{r}jC_j.
	\end{align}

	(b) To fail Condition \eqref{cond2}, we can obtain $\alpha$-forward switchings by choosing   $(2k-1)$ edges arbitrarily, and then  choose one  edge $e_i$ or $e_i'$ that is at most of distance $2r-1$ from $\tilde {e}_i$ for some $1\leq i\leq k$. From the degree constraints, the number of edges of distance less than $2r$ from some edge is at most $O((d_1-1)^r(d_2-1)^r)$.  Similar to Part (a), by considering whether $e_i$ or $e_i'$ is chosen for $1\leq i\leq k$,
	the number of such switchings  is at most 
	\begin{align}
&\left([n-1]_{k-1}[m]_kd_1^{k-1}d_2^k +[n]_{k}[m-1]_kd_1^{k}d_2^{k-1} \right)\cdot k \cdot  O((d_1-1)^r(d_2-1)^r) \notag\\
=&\frac{1}{nd_1}	[n]_{k}[m]_k(d_1d_2)^k k O((d_1-1)^r(d_2-1)^r).	\label{eq:abcd2}
	\end{align}
	
	(c)  For Condition \eqref{cond3}, there are three cases to consider depending on  whether the pair is $(e_i,e_j)$, $(e_i',e_j')$ or $(e_i,e_j')$.
	
	Suppose the pair $e_i,e_j$ violates Condition \eqref{cond3}. We pick the pair of edges  that are within distance $r-1$ and pick the remaining $(2k-2)$ edges arbitrarily. There are $(nd_1)$ many ways to choose $e_i$. When $e_i$ is fixed,  there are at most 
	$ O((d_1-1)^{(r+1)/2} (d_2-1)^{(r+1)/2})$ choices for $e_j$. Hence the number of switchings that fail Condition \eqref{cond3} is at most 
  \begin{align*}
   &   (nd_1)\cdot [    O((d_1-1)^{(r+1)/2}(d_2-1)^{(r+1)/2})] \cdot k(k-1)	\cdot ([n-2]_{k-2} [m]_k d_1^{k-2} d_2^k)\\
  =&\frac{1}{nd_1}[n]_k[m]_kd_1^{k}d_2^k \cdot  k^2\cdot  O\left((d_1-1)^{(r+1)/2}(d_2-1)^{(r+1)/2}\right). 
  \end{align*}
   By the same argument, if the pair is $(e_i',e_j')$, the number of switchings that fail Condition \eqref{cond3} is at most
 \begin{align*}
%  &[n]_{k}d_1^{k} [m-2]_{k-2}d_2^{k-2}\cdot (nd_1)\cdot d_2\cdot  k^2\cdot O((d_1-1)^{(r+1)/2}(d_2-1)^{(r+1)/2}) 	\\
 \frac{1}{nd_1}[n]_k[m]_kd_1^kd_2^{k} k^2\cdot O\left((d_1-1)^{(r+1)/2}(d_2-1)^{(r+1)/2}\right).
 \end{align*}
  When the two edges of the pair violating Condition \eqref{cond3} are $e_i,e_j'$ for some $i,j$, the number is at most  
	\begin{align*}
	&(nd_1)\cdot [  O((d_1-1)^{(r+1)/2}(d_2-1)^{(r+1)/2})]\cdot (2k^2)\cdot ([n-1]_{k-1}[m-1]_{k-1}d_1^{k-1} d_2^{k-1}) \\
	=&\frac{1}{nd_1}[n]_k[m]_kd_1^kd_2^{k} k^2O\left((d_1-1)^{(r+1)/2}(d_2-1)^{(r+1)/2}\right).
	\end{align*}
% 	 Recall $d_1\geq d_2$ and $k\leq r$, we have
% 	\[ r\left(  \frac{d_1+d_2}{\sqrt{(d_1-1)(d_2-1)}}+1\right)=O((d_1-1)^{r/2-1}(d_2-1)^{r/2-1}). \] 
	Combining the three cases in Part (c),  the number of switchings that violate Condition \eqref{cond3} is at most
	\begin{align}\label{eq:abcd3}
% 	&k\left(  \frac{d_1+d_2}{\sqrt{(d_1-1)(d_2-1)}}+1\right) \cdot \frac{1}{nd_1}[n]_k[m]_k(d_1d_2)^kk\cdot O\left((d_1-1)^{r/2+1}(d_2-1)^{r/2+1}\right)	\\
	\frac{1}{nd_1}[n]_k[m]_k(d_1d_2)^kk^2\cdot O\left((d_1-1)^{(r+1)/2}(d_2-1)^{(r+1)/2}\right).
	\end{align}

	(d)  Since the distance between a pair of vertices in $V_1$ or $V_2$ must be even, to violate Condition \eqref{cond4}, we can choose a pair $u_{i},u_{i}'\in V_1$ or $v_{i},v_{i}'\in V_2$ that are within distance $2r-2$ first, then choose other edges arbitrarily. Similar to the cases above, the number of  switchings that fail  Condition \eqref{cond4} is  at most   
	\begin{align}\label{eq:abcd4}
% 	& [n-2]_{k-2}d_1^{k-2}[m]_kd_2^{k}\cdot nd_1^2k\cdot O((d_1-1)^{r-1}(d_2-1)^{r-1})\\
% 	&+ [n-2]_{k-2}d_1^{k-2}[m]_kd_2^{k}\cdot md_2^2k\cdot O((d_1-1)^{r-1}(d_2-1)^{r-1}) \\
% 	&+ [n]_{k}d_1^{k}[m-2]_{k-2}d_2^{k-2}\cdot nd_1^2k\cdot O((d_1-1)^{r-1}(d_2-1)^{r-1})\\
% 	&+[n]_{k}d_1^{k}[m-2]_{k-2}d_2^{k-2}\cdot md_2^2k\cdot O((d_1-1)^{r-1}(d_2-1)^{r-1})	\\
	  \frac{1}{nd_1}[n]_k[m]_kd_1^{k}d_2^kkO((d_1-1)^{r}(d_2-1)^{r}).
	\end{align}

Combining the $4$ Cases (a)-(d) above, from \eqref{eq:abcd1}, \eqref{eq:abcd2},\eqref{eq:abcd3}, and \eqref{eq:abcd4}, we have at most
\begin{align*}
\frac{4}{nd_1}[n]_{k}[m]_k (d_1d_2)^{k} \left(k\sum_{j=2}^r j C_j+ O(k(d_1-1)^{r}(d_2-1)^r)	\right)
\end{align*}
many switchings that fail one of the Conditions \eqref{cond1}-\eqref{cond4} among the  $[n]_{k}[m]_k(d_1d_2)^{k}$ possible switchings. Then for an absolute constant $c_1>0$, 
\begin{align*}
F_{\alpha}\geq [n]_{k}[m]_kd_1^{k}d_2^k\left(1-\frac{4k\sum_{j=2}^r jC_j+c_1k(d_1-1)^{r}(d_2-1)^{r}}{nd_1} \right).	
\end{align*}
Therefore \eqref{eq:lowerF} holds.
\end{proof}

For the number of backward switchings, we obtain a similar upper bound, but the lower bound is only in expectation.
\begin{lemma}\label{lem:B}
Let $G$ be a random $(d_1,d_2)$-biregular bipartite graph	and let $\alpha$ be a cycle of length $2k\leq 2r$ in the complete bipartite graph $K_{n,m}$. Let $B_{\alpha}$ be the number of valid backward switchings from $G$ that create $\alpha$. Then 
\begin{align}\label{upperB}
B_{\alpha}\leq (d_1(d_1-1))^k(d_2(d_2-1))^k,
\end{align}
and there is an absolute constant $c_2>0$ such that
\begin{align}\label{lowerB}
\mathbb E B_{\alpha}\geq 	(d_1(d_1-1))^k(d_2(d_2-1))^k\left(1- \frac{c_2k(d_1-1)^{r}(d_2-1)^{r}}{nd_1}\right).
\end{align}
\end{lemma}

\begin{proof}
	Given $\alpha$, from the degree constraints, the number of choices for $u_i, u_i',v_i, v_i', 1\leq i\leq k$ that yield a valid backward $\alpha$ switching is at most $(d_1(d_1-1))^k(d_2(d_2-1))^k$, which gives \eqref{upperB}.

	 For  the lower bound, we consider the quantity $B:=\sum_{\beta} B_{\beta}$, where $\beta$ is summing over all possible  cycles of length $2k$ in the complete bipartite graph $K_{n,m}$. 
	As in the proof of Lemma \ref{lem:F}, we give conditions that guarantee a valid backward switching.
	
	Assume $\beta=(x_1,y_1,\cdots, x_{k},y_{k})$. We first consider backward switchings that create $\beta$.	 
		Suppose the paths $v_{i}x_iv_{i}', u_{i}y_iu_{i}', 1\leq i\leq k$ in $G$ satisfy the following conditions:
	\begin{enumerate}
		\item The edges $x_iv_{i}, x_iv_{i}', y_iu_{i}, $ and $y_iu_{i}'$ are not contained in any short cycles.\label{conda}
		\item  For $1\leq i\leq k$, the distance between any vertex in the path $v_ix_iv_i'$ and any vertex in the path $u_{i}y_{i}u_{i}'$ is at least $2r$.\label{condb}
		\item For all $1\leq i\leq k$ and $1\leq j\leq k/2$, the distance between the paths $v_ix_iv_i'$ and $v_{i+j}x_{i+j}v_{i+j}'$ (the index $i+j$ is calculated modulo $k$) and  the distance between  $u_iy_iu_i'$ and $u_{i+j}y_{i+j}u_{i+j}'$ are at least $2r-2j+1$. \label{condc}
		\item For $1\leq i\leq k, 1\leq j\leq k/2$, the distance between $v_ix_iv_i'$ and $u_{i+j}y_{i+j}u_{i+j}'$, and the distance between $u_iy_iu_i'$ and $v_{i+j}x_{i+j}v_{i+j}'$ are at least $2r-2j+2$.  \label{condd}
	\end{enumerate} 	
We will show the four conditions above guarantee a valid backward $\beta$-switching.

By Condition \eqref{conda}, no short cycles are deleted. We denote $x\not\sim y$ if two vertices $x,y$ are not connected in $G$.
An immediate consequence of Condition \eqref{condb} ensures that $x_i\not\sim y_{i}$ and $u_i\not\sim v_i$, $u_i'\not\sim v_i'$, and Condition \eqref{condd} ensures that $y_i\not\sim x_{i+1}$. Therefore such  switching can be applied. 

Let $G'$ be the graph obtained by applying the backward $\beta$-switching. We need to check that no short cycles other than $\beta$ are created in $G'$.

Suppose a short cycle $\beta'\not=\beta$ is created. Then $\beta'$ possibly consists of paths in $G\cap G'$, portions of $\beta$, and edges $u_iv_i,u_i'v_i'$ for some $1\leq i\leq k$. 
Any such path in $G\cap G'$ must have length at least $r$ because
	\begin{itemize}
		\item  if it starts in one of the  sets $\{x_i,v_{i},v_{i}'\}$ or $\{y_i, u_{i},u_{i}'\}$ for $1\leq i\leq k$, and ends at a different set $\{x_j,v_{j},v_{j}'\}$ or $\{y_j, u_{j},u_{j}'\}$ for $1\leq j\leq k$, then Conditions \eqref{condb}, \eqref{condc} and \eqref{condd} imply this;
		\item if it starts and ends in the same set $\{x_i,v_i,v_i'\}$ or $\{y_i,u_i,u_i'\}$, then it follows from Condition \eqref{conda} that the path must have length at least $r$.
	\end{itemize}
	It follows that $\beta'$ must contain exactly one such path, otherwise, if two such paths are included in $\beta'$, the length of $\beta'$  is greater than $2r$, a contradiction to the fact that $\beta'$ is a short cycle. 
	
	Besides this path in $G\cap G'$, the remainder of $\beta'$ must either be an edge $u_iv_i$ or $u_i'v_i'$, or a portion of $\beta$. 
		If the remainder is some $u_iv_i$, then the distance between $u_i$ and $v_i$ in $G$ is at most $2r-1$, a contradiction to Condition \eqref{condb}. The same holds if the remainder is some $u_i'v_i'$.

	If the remainder is a portion of $\beta$, then there exist two vertices in $\beta$ connected by the path in $G\cap G'$ contained in $\beta'$. If the two vertices are $x_i, x_{i+j}$ for some $1\leq i\leq k, 1\leq j\leq k/2$, then from Condition \eqref{condc}, the path in $G\cap G'$ contained in  $\beta'$ that connects the two vertices has length at least $2r-2j+1$. Since the path in $\beta$ connecting $x_i,x_{i+j}$ has length $2j$, this implies $\beta'$ has length at least 
	$(2r-2j+1)+2j=2r+1,$ a contradiction. In the same way, if the two vertices are $y_i, y_{i+j}$ for some $1\leq i\leq k, 1\leq j\leq k/2$, we can find a contradiction for $\beta'$ from Condition \eqref{condc}. 
	
	If the two vertices connected by the path  are $x_i,y_{i+j}$ with $1\leq i\leq k, 1\leq j\leq k/2$, then the path in $\beta$ connecting the two vertices has length at least $2j-1$. Combining the path in $G\cap G'$ contained in $\beta'$, from Condition \eqref{condd}, we conclude that $\beta'$ has length at least $(2r-2j+2)+(2j-1)=2r+1,$ a contradiction. By the same argument, if the two vertices connected by the path are $y_{i},x_{i+j}$ for some $1\leq i\leq k, 1\leq j\leq k/2$, we can find a contradiction that $\beta'$ is not a short cycle.

	Therefore such $\beta'$ does not exist, and all backward switchings satisfying Conditions \eqref{conda}-\eqref{condd} are valid.

There are $ {[n]_k[m]_k}/(2k)$ choices for the $2k$-cycle $\beta$ in the complete bipartite graph $K_{n,m}$, and at most $(d_1(d_1-1))^k(d_2(d_2-1))^k$ choices for $u_i,u_i',v_i, v_i', 1\leq i\leq k$ given $\beta$. We now count how many possible backward switchings  violate one of the four Conditions \eqref{conda}-\eqref{condd} to get a lower bound on $B$.  We treat the  Conditions \eqref{conda}-\eqref{condd} in 4 parts.

(a) Suppose Condition \eqref{conda} is violated. We estimate the number of switchings by choosing one edge from the set of  edges in short cycles and the other edges arbitrarily. Note that by our definition of switchings, we identify $2k$ different switchings by applying the cyclic rotation or inversion on $[k]$. Suppose we choose an edge $x_iv_{i}$ or $x_iv_{i}'$ from short cycles, similar to the analysis in Lemma \eqref{lem:F}, the number of switchings is at most 
	\begin{align*}
		&\left(2k\sum_{j=2}^r 2jC_j
\cdot (d_1-1)\right)\cdot \left(\frac{1}{2k}[n-1]_{k-1} (d_1(d_1-1))^{k-1}\cdot [m]_k(d_2(d_2-1))^k\right)\\
=&\frac{2}{nd_1}[n]_k[m]_k[d_1(d_1-1)d_2(d_2-1)]^{2k} \sum_{j=2}^r jC_j.
	\end{align*}
Similarly, if we choose an edge $y_iu_{i}$ or $y_iu_{i}'$ from short cycles,  the number of switchings is at most 
\begin{align*}
    &\left(2k\sum_{j=2}^r 2jC_j
\cdot (d_2-1)\right)\cdot \left(\frac{1}{2k}[m-1]_{k-1}(d_2(d_2-1))^{k-1}\cdot [n]_{k} (d_1(d_1-1))^k\right)\\
=&\frac{2}{nd_1}[n]_k[m]_k[d_1(d_1-1)d_2(d_2-1)]^{2k} \sum_{j=2}^r jC_j.
\end{align*}
  Combining two parts, the number of switchings that violate Condition \eqref{conda} is at most
  \begin{align}\label{eq:CASE1}
  	& \frac{8k}{nd_1}[n]_{k} [m]_{k}[d_1(d_1-1) d_2(d_2-1)]^{k} \sum_{j=2}^r jC_j. 
  \end{align}

(b) Suppose for some $1\leq i\leq k$, two paths $v_ix_iv_i'$ and $u_iy_iu_i'$ are within distance $2r-1$. The number of switching is at most
\begin{align}
    &\frac{[n-1]_{k-1}[m-1]_{k-1}}{2k}[d_1(d_1-1)d_2(d_2-1)]^{k-1}\cdot (knd_1(d_1-1))\cdot O((d_1-1)^{r}(d_2-1)^{r+1}) \notag\\
    =&\frac{1}{nd_1}[n]_k[m]_k(d_1(d_1-1))^k(d_2(d_2-1))^kO((d_1-1)^{r}(d_2-1)^{r}).\label{eq:suppose1}
\end{align}

(c) Suppose for some $1\leq i\leq k,1\leq j\leq k/2$, two paths  $\{v_{i}x_iv_{i}', v_{i+j}x_{i+j}v_{i+j}'\}$  are within distance $2r-2j$. The number of switchings is at most
 \begin{align*}
 &\frac{[n-2]_{k-2}[m]_k}{2k}(d_1(d_1-1))^{k-2}(d_2(d_2-1))^{k}\cdot nd_1(d_1-1) \sum_{i=1}^k\sum_{j=1}^{\lfloor k/2\rfloor}O((d_1-1)^{r-j+2}(d_2-1)^{r-j})\\
 =&\frac{1}{nd_1}[n]_k[m]_k(d_1(d_1-1))^k(d_2(d_2-1))^kO((d_1-1)^{r}(d_2-1)^{r-1}).
 \end{align*}
Suppose for some $1\leq i\leq k, 1\leq j\leq k/2$, two paths $\{u_{i}y_iu_{i}', u_{i+j}y_{i+j}u_{i+j}'\}$ are within distance $2r-2j$. Similarly, the number of switchings is bounded by
 \begin{align*}
 \frac{1}{nd_1}[n]_k[m]_k(d_1(d_1-1))^k(d_2(d_2-1))^kO((d_1-1)^{r}(d_2-1)^{r-1}).
 \end{align*}
Therefore the number of switchings that violate Condition \eqref{condc} is at most
  \begin{align}\label{eq:suppose2}
 \frac{1}{nd_1}[n]_k[m]_k(d_1(d_1-1))^k(d_2(d_2-1))^kO((d_1-1)^{r}(d_2-1)^{r-1}).
 \end{align}

(d) Suppose two paths $v_{i}x_iv_{i}', u_{i+j}y_{i+j}u_{i+j}'$ for some $1\leq i\leq k, 1\leq j\leq k/2$  are within distance  $2r-2j+1$. The number of choices  is at most 
 \begin{align*}
%  &[n]_{k}[m-1]_{k-1}(d_1(d_1-1))^{k}(d_2(d_2-1))^{k-1} O((d_1-1)^{r}(d_2-1)^{r-1})\cdot d_2(d_2-1)\\
 &\frac{[n]_k[m-1]_{k-1}}{2k}(d_1(d_1-1))^{k}(d_2(d_2-1))^{k-1} \sum_{i=1}^k\sum_{j=1}^{\lfloor k/2 \rfloor}O((d_1-1)^{r-j+1}(d_2-1)^{r-j+2}\\
 =&\frac{1}{nd_1}[n]_k[m]_k(d_1(d_1-1))^k(d_2(d_2-1))^kO\left((d_1-1)^{r}(d_2-1)^{r}\right).
 \end{align*}
Suppose two paths $u_{i}y_iu_{i}', v_{i+j}x_{i+j}v_{i+j}'$ for some $1\leq i\leq k, 1\leq j\leq k/2$  are within distance  $2r-2j+1$. By the same argument, the number of choices  is at most
 \[\frac{1}{nd_1}[n]_k[m]_k(d_1(d_1-1))^k(d_2(d_2-1))^kO\left((d_1-1)^{r}(d_2-1)^{r}\right).\]
 Then the number of switchings that violate Condition \eqref{condd} is at most
 \begin{align}\label{eq:suppose3}
     \frac{1}{nd_1}[n]_k[m]_k(d_1(d_1-1))^k(d_2(d_2-1))^kO\left((d_1-1)^{r}(d_2-1)^{r}\right).
 \end{align}
 
From \eqref{eq:CASE1}, \eqref{eq:suppose1}, \eqref{eq:suppose2} and \eqref{eq:suppose3}, the lower bound of $B$ is given by
 \begin{align}\label{eq:lowerB}
 B\geq \frac{	[n]_k[m]_k}{2k}(d_1(d_1-1))^k(d_2(d_2-1))^k\left( 1-\frac{8k  \sum_{j=2}^r jC_j+O(k(d_1-1)^r(d_2-1)^r)}{nd_1}\right).
 \end{align}
 
 By Lemma \ref{lem:Mckay} (b),
 \[ \mathbb EC_k\leq \frac{[n]_k[m]_k}{2k}\frac{c_1(d_1-1)^k(d_2-1)^k}{n^k m^k}\leq \frac{c_1(d_1-1)^k(d_2-1)^k}{2k}.\]
 Applying the inequality above to \eqref{eq:lowerB}, we obtain
 \begin{align*}
 \mathbb EB\geq \frac{[n]_k[m]_k}{2k}(d_1(d_1-1))^k(d_2(d_2-1))^k\left( 1-\frac{O(k(d_1-1)^r(d_2-1)^r)}{nd_1}\right).
 \end{align*}
 
By the exchangeability of the vertex labels in the uniformly distributed RBBG model, the law of $B_{\beta}$ is the same for any $2k$-cycle $\beta$. Then
\begin{align*}
\mathbb EB_{\alpha}=\frac{2k}{[n]_k[m]_k}\mathbb E B\geq (d_1(d_1-1)d_2(d_2-1))^k\left(1-\frac{c_2k(d_1-1)^{r}(d_2-1)^r}{nd_1}\right), 	
\end{align*} for an absolute constant $c_2>0$. This completes the proof.
\end{proof}

\subsection{Poisson approximation of cycle counts}\label{sec:Poi}

In this section, we prove the cycle counts in RBBGs are asymptotically distributed as Poisson random variables. The main tool we will use is the following  total variation distance bound from \cite{chatterjee2005exchangeable}.

\begin{lemma}[Proposition 10 in \cite{chatterjee2005exchangeable}]\label{lem:exchangeablepair}
Let $W=(W_1,\dots,W_r)$ be a random vector taking values in $\mathbb N^r$, and let the coordinates of $Z=(Z_1,\dots,Z_r)$ be independent Poisson random variables with $\mathbb EZ_k=\mu_k$. Let $W'=(W_1',\dots, W_r')$ be defined on the same space as $W$, with $(W,W')$ an exchangeable pair. For any choice of $\sigma$-algebra $\mathcal F$ with respect to which $W$ is measurable and any choice of constants $c_k$, we have
\begin{align}
d_{\textnormal{TV}}(W,Z)\leq \sum_{k=1}^r\xi_k\left( \mathbb E|\mu_k-c_k\mathbb P(\Delta_k^+\mid \mathcal F)|+\mathbb E|W_k-c_k\mathbb P(\Delta_k^-\mid \mathcal F)|\right),	
\end{align}
where $\xi_k:=\min \{ 1,1.4\mu_k^{-1/2}\}$ and 
\begin{align}\label{eq:eventW}
	\Delta_k^+:=&\{W_k'=W_{k}+1, W_j=W_j', k<j\leq r\},\\
	\Delta_k^-:=&\{W_k'=W_{k}-1, W_j=W_j', k<j\leq r\}.
\end{align}
\end{lemma}

We apply Stein's method to obtain the following Poisson  approximation in total variation distance.

\begin{theorem}\label{thm:Poissoncount}
Let $G$ be a random $(d_1,d_2)$-biregular bipartite graph with cycle counts $(C_k, k\geq 2)$. Let $(Z_k, k\geq 2)$ be independent Poisson random variables with \[\mu_k:=\mathbb EZ_k=\frac{(d_1-1)^{k}(d_2-1)^{k}}{2k}.\]	 For any $n,m\geq 1$ and $r\geq 2, d_1\geq 3$, there exists an absolute constant $c_6>0$ such that 
\[ d_{\textnormal{TV}}((C_2,\dots,C_r),(Z_2,\dots,Z_r))\leq  \frac{c_6\sqrt{r}(d_1-1)^{3r/2}(d_2-1)^{3r/2}}{nd_1}.\]
\end{theorem}

\begin{proof} 
If $d_1>n^{1/3}$ or $r>n^{1/10}$, then \[\frac{c_6\sqrt{r}(d_1-1)^{3r/2}(d_2-1)^{3r/2}}{nd_1
}>1\] for a sufficiently large choice of $c_6$ and the theorem holds trivially. Thus we assume $d_1\leq n^{1/3}$ and $r\leq n^{1/10}$.
We now construct an exchangeable pair of random biregular bipartite graphs by taking a step in a reversible Markov chain.

Define a graph $\mathcal G$ whose vertex set consists of all $(d_1,d_2)$-biregular bipartite graphs.	
 If there is a  valid forward  or backward $\alpha$-switching  from a $(d_1,d_2)$-biregular bipartite graph $G_0$ to another graph $G_1$ with the length of $\alpha$ being $2k$, we make an undirected edge in $\mathcal G$ between $G_0,G_1$ and place a weight of \[\frac{1}{[n]_{k}[m]_{k}(d_1d_2)^{k}}\] on each such edge.  Define the degree of a vertex in $\mathcal G$ to be the sum of weights from all adjacent edges. Let $d_0$ be the largest degree in $\mathcal G$. To make $\mathcal G$ regular, we add a weighted loop to each vertex if necessary to increase the degree of all vertices to $d_0$. 

Now consider the simple random walk on $\mathcal G$. This is a reversible Markov chain with respect to the uniform distribution on $(d_1,d_2)$-biregular bipartite graphs. Thus suppose $G$ is a uniformly chosen random biregular bipartite graph, we can obtain another random biregular bipartite graph $G'$ by taking an extra step in the random walk from $G$, and the pair $(G,G')$ is exchangeable.

Let $\mathcal J_k$ be the collection of cycles of length $2k$ in $K_{n,m}$ with $k\leq r$. We have  $|\mathcal J_k|= [n]_k[m]_k/{2k}$. Define $I_{\alpha}=\mathbf{1}\{\alpha \subseteq G\}$.
Then  $C_k=\sum_{\alpha\in \mathcal J_k}\mathbf{1}_{\alpha}.$ Let $I_{\alpha}', C_k'$ be defined  on $G'$ in the same way. Since $G$ and $G'$ are exchangeable, the vectors $(C_2,\dots,C_r)$ and $(C_2',\dots,C_r')$ are also exchangeable. We can then apply Lemma \ref{lem:exchangeablepair} to this exchangeable pair of vectors. Now define two events
\begin{align*}
	\Delta_{k}^+:=&\{C_k'=C_k+1, C_j=C_j', k<  j\leq r\},\\
	\Delta_{k}^-:=&\{C_k=C_k'+1, C_j=C_j', k<  j\leq r\}.
\end{align*}
%The above events are similar to the event \eqref{eq:eventW} in the following sense. Suppose we arrange all elements in $\mathcal J$ and $\alpha$ is the $k$-th element. Then $\Delta_k^+$ is the event that $I_{\alpha}'=I_{\alpha}+1$ and for all $\beta$ arranged after $\alpha$ we have $I_{\beta}'=I_{\beta}+1$. However, by our construction for all $\beta\not=\alpha$, $\beta$ cannot be destroyed or created, therefore $\Delta_{\alpha}^+=\Delta_k^+$.
Through our construction of the exchangeable pair,
\begin{align*}
\mathbb P(\Delta_{k}^+\mid G)&=\sum_{\alpha\in \mathcal J_k}\frac{B_{\alpha}}{d_0[n]_{k}[m]_{k}(d_1d_2)^{k}}	,\\ 
\mathbb P(\Delta_{k}^-\mid G)&=\sum_{\alpha\in\mathcal J_k}\frac{F_{\alpha}}{d_0[n]_{k}[m]_{k}(d_1d_2)^{k}}.
\end{align*}

Applying Lemma \ref{lem:exchangeablepair} with all $c_k=d_0, 1\leq k\leq r$, we have 
\begin{align}
 &d_{\textnormal{TV}}((C_3,\dots,C_r),(Z_3,\dots,Z_r)) \notag\\
  \leq & \sum_{k=2}^r\xi_k\mathbb E\left| \mu_k-\sum_{\alpha\in \mathcal J_k}\frac{B_{\alpha}}{[n]_{k}[m]_{k}(d_1d_2)^{k}}\right|+	\sum_{k=2}^r\xi_k \mathbb E\left| C_k-\sum_{\alpha\in\mathcal J_k}\frac{F_{\alpha}}{[n]_{k}[m]_{k}(d_1d_2)^{k}}\right| \notag\\
  = & \sum_{k=2}^r\xi_k\mathbb E\left| \sum_{\alpha\in \mathcal J_k} \frac{(d_1-1)^k(d_2-1)^k}{[n]_k[m]_k}-\frac{B_{\alpha}}{[n]_{k}[m]_{k}(d_1d_2)^{k}}\right|+	\sum_{k=2}^r\xi_k \mathbb E\left| \sum_{\alpha\in\mathcal J_k}I_{\alpha}-\frac{F_{\alpha}}{[n]_{k}[m]_{k}(d_1d_2)^{k}}\right| \notag\\
  \leq & \sum_{k=2}^r\xi_k \left(\sum_{\alpha\in \mathcal J_k}\mathbb E\left| \frac{(d_1-1)^k(d_2-1)^k}{[n]_k[m]_k}-\frac{B_{\alpha}}{[n]_{k}[m]_{k}(d_1d_2)^{k}}\right|+ \sum_{\alpha\in\mathcal J_k}\mathbb E\left| I_{\alpha}-\frac{F_{\alpha}}{[n]_{k}[m]_{k}(d_1d_2)^{k}}\right| \right). \label{eq:twosum}
  \end{align}
  
  For the rest of the proof, we estimate the following two sums
  \begin{align}
 &\sum_{\alpha\in \mathcal J_k}\mathbb E\left| \frac{(d_1-1)^k(d_2-1)^k}{[n]_k[m]_k}-\frac{B_{\alpha}}{[n]_{k}[m]_{k}(d_1d_2)^{k}}\right|\label{eq:twosum1},\\
 &\sum_{\alpha\in\mathcal J_k}\mathbb E\left| I_{\alpha}-\frac{F_{\alpha}}{[n]_{k}[m]_{k}(d_1d_2)^{k}}\right|  \label{eq:twosum2}
\end{align}
from \eqref{eq:twosum} in different ways. 

(1) \textit{The upper bound on \eqref{eq:twosum1}}.
From Lemma \ref{lem:B}, for all $\alpha\in \mathcal J_k$, 
%\[B_{\alpha}\leq (d_1(d_1-1))^k(d_2(d_2-1))^k\] and
%\[\mathbb E B_{\alpha}\geq 	(d_1(d_1-1))^k(d_2(d_2-1))^k\left(1- \frac{c_2r(d_1-1)^{r-1}(d_2-1)^{r-1}}{n}\right).\]
\begin{align*}
\mathbb E\left| \frac{(d_1-1)^k(d_2-1)^k}{[n]_k[m]_k}-\frac{B_{\alpha}}{[n]_{k}[m]_{k}(d_1d_2)^{k}}\right|&=	\frac{(d_1-1)^k(d_2-1)^k}{[n]_k[m]_k}-\frac{\mathbb EB_{\alpha}}{[n]_{k}[m]_{k}(d_1d_2)^{k}} \notag\\
&\leq \frac{c_2k(d_1-1)^{r+k}(d_2-1)^{r+k}}{nd_1[n]_k[m]_k},
\end{align*}
where the first line is from \eqref{upperB} and the second line is from \eqref{lowerB}.
Therefore   \eqref{eq:twosum1} satisfies
\begin{align}\label{eq:Balpha}
 \sum_{\alpha\in \mathcal J_k}\mathbb E\left| \frac{(d_1-1)^k(d_2-1)^k}{[n]_k[m]_k}-\frac{B_{\alpha}}{[n]_{k}[m]_{k}(d_1d_2)^{k}}\right|	\leq  \frac{c_2[(d_1-1)(d_2-1)]^{r+k}}{2nd_1}.
\end{align}

(2) \textit{The upper bound on \eqref{eq:twosum2}}.
To bound the  summation in \eqref{eq:twosum2}, for a given short cycle $\alpha$, we consider a partition of $\mathcal G$ in the following way:
\begin{align*}
A_1^{\alpha} &= \{ \text{$G$ does not contain $\alpha$}\},\\
A_2^{\alpha} &=\{ \text{$G$  contains $\alpha$, which does not share an edge with another short cycle in $G$}\},\\
A_3^{\alpha} &=\{ \text{$G$ contains $\alpha$, which shares an edge with another short cycle in $G$}\}.	
\end{align*}

Conditioned on $A_1^{\alpha}$, we have  $I_{\alpha}=F_{\alpha}=0$. Conditioned on $A_2^{\alpha}$,  both the upper and lower bounds in Lemma \ref{lem:F} can apply, which yield the following inequality:
\begin{align}
	\left|I_{\alpha}-\frac{F_{\alpha}}{[n]_{k}[m]_{k}(d_1d_2)^{k}}\right|&\leq \frac{4k\sum_{j=2}^rj C_j+c_1k(d_1-1)^r(d_2-1)^r}{nd_1}.
\end{align}
Conditioned on $A_3^{\alpha}$, we have $I_{\alpha}=1,F_{\alpha}=0$. 

With the partition of $\mathcal G$, the following inequality holds:
\begin{align}\label{eq:Ia}
\mathbb  E\left|I_{\alpha}-\frac{F_{\alpha}}{[n]_{k}[m]_{k}(d_1d_2)^{k}}\right| &=  \mathbb  E\left[\mathbf{1}_{A_2^{\alpha}}\left|I_{\alpha}-\frac{F_{\alpha}}{[n]_{k}[m]_{k}(d_1d_2)^{k}}\right|\right]+\mathbb P(A_3^{\alpha})\notag\\
&\leq \frac{2k}{nd_1}\mathbb E\left[\mathbf{1}_{A_2^{\alpha}}\sum_{j=2}^r2jC_j\right]+\frac{c_1k(d_1-1)^r(d_2-1)^r} {nd_1}\mathbb P(A_2^{\alpha})+\mathbb P(A_3^{\alpha}).
\end{align}

Let $\mathcal J_{\alpha}$ be the set of all short cycles in $K_{n,m}$ that share no edges with $\alpha$. On the event $A_2^{\alpha}$, the graph $G$ contains no short cycles outside $\mathcal J_{\alpha}$ except for $\alpha$. Define $|\beta|$ be the length of the cycle $\beta$. Then 
\[ \sum_{j=2}^r 2jC_j=2k+\sum_{\beta\in \mathcal J_{\alpha}} |\beta|I_{\beta}.\]
Therefore the right-hand side of \eqref{eq:Ia} can be bounded by
\begin{align}
	& \frac{4k^2}{nd_1}\mathbb P (A_2^{\alpha})+\frac{2k}{nd_1}\sum_{\beta\in\mathcal J_{\alpha}} |\beta|\mathbb EI_{\alpha}I_{\beta}+\frac{c_1k(d_1-1)^r(d_2-1)^r} {nd_1}\mathbb P(A_2^{\alpha})+\mathbb P(A_3^{\alpha}) \notag\\
	\leq  &\frac{4k^2}{nd_1}\mathbb P (\alpha\subseteq G)+\frac{c_1k(d_1-1)^r(d_2-1)^r} {nd_1}\mathbb P(\alpha\subseteq G)+\frac{2k}{nd_1}\sum_{\beta\in\mathcal J_{\alpha}} |\beta|\mathbb EI_{\alpha}I_{\beta}+\mathbb P(A_3^{\alpha}).\label{eq:4444}
\end{align}
By Lemma \ref{lem:Mckay}(1),
\begin{align*}
	\frac{4k^2}{nd_1}\mathbb P (\alpha\subseteq G) &=O\left( \frac{k^2[(d_1-1)(d_2-1)]^{k}}{nd_1(nm)^k}\right),\\
	\frac{c_1k(d_1-1)^r(d_2-1)^r} {nd_1}\mathbb P(\alpha\subseteq G) &=O\left( \frac{k[(d_1-1)(d_2-1)]^{k+r}}{nd_1(nm)^k}\right).
\end{align*}
Hence the first and the second term in \eqref{eq:4444} combine to yield a corresponding upper bound in \eqref{eq:twosum2} of
\begin{align}\label{eq:sumbound1}
\sum_{\alpha\in\mathcal J_k}\left(\frac{4k^2}{nd_1}\mathbb P (\alpha\subseteq G)+\frac{c_1k(d_1-1)^r(d_2-1)^r} {nd_1}\mathbb P(\alpha\subseteq G) \right)= O\left( \frac{[(d_1-1)(d_2-1)]^{k+r}}{nd_1}\right).
\end{align}

From Lemma \ref{lem:Mckay} (3), we have for any $\beta\in \mathcal J_{\alpha}$ with $|\beta|=2j$, \[\mathbb EI_{\alpha}I_{\beta}=\mathbb P(\alpha\cup \beta \in G)\leq \frac{c_1[(d_1-1)(d_2-1)]^{j+k}}{(nm)^{j+k}}.\]
For $2\leq j\leq r$, there are at most $[n]_j[m]_j/(2j)$ cycles in $\mathcal J_{\alpha}$ of length $2j$. The third term in \eqref{eq:4444}  then satisfies
\begin{align*}
	\frac{2k}{nd_1}\sum_{\beta\in\mathcal J_{\alpha}} |\beta|\mathbb EI_{\alpha}I_{\beta}&\leq \frac{2k}{nd_1}\sum_{j=2}^r\frac{[n]_j[m]_j}{2j}\cdot 2j\cdot \frac{c_1[(d_1-1)(d_2-1)]^{j+k}}{(nm)^{j+k}}\\
	&=O\left( \frac{k[(d_1-1)(d_2-1)]^{r+k}}{nd_1(nm)^{k}}\right).
\end{align*}
Summing over all possible $\alpha\in \mathcal J_k$, we obtain a corresponding term in \eqref{eq:twosum2} of
\begin{align}\label{eq:sumbound2}
	\sum_{\alpha\in \mathcal J_k}\frac{2k}{nd_1}\sum_{\beta\in\mathcal J_{\alpha}} |\beta|\mathbb EI_{\alpha}I_{\beta}&=O\left( \frac{[(d_1-1)(d_2-1)]^{r+k}}{nd_1}\right).
\end{align}

Now given \eqref{eq:sumbound1} and \eqref{eq:sumbound2}, to control \eqref{eq:twosum2},   it remains to  estimate $\sum_{\alpha\in \mathcal J_k}\mathbb P(A_3^{\alpha})$. Let $\mathcal K_{\alpha}$ be the set of all short cycles in $K_{n,m}$ that share an edge with $\alpha$, not including $\alpha$ itself. By a union bound,
\begin{align}\label{eq:A3}
	\sum_{\alpha\in \mathcal J_k}\mathbb P(A_3^{\alpha})\leq \sum_{\alpha\in \mathcal J_k}\sum_{\beta\in \mathcal K_{\alpha}}\mathbb  P(\alpha\cup\beta\subset G). 
\end{align}

From  \eqref{eq:McKay2} in Lemma \ref{lem:Mckay}, the upper bound for $\mathbb  P(\alpha\cup\beta\subset G)$ depends on the lengths of $\alpha, \beta$, and the number of edges that $\alpha,\beta$ share. To get an upper bound on \eqref{eq:A3}, we will classify and count the number of pairs $(\alpha,\beta)$ based on the structure of $\alpha\cup\beta$.

Recall $\alpha$ has length $2k$. Suppose $\beta$ has length $2j$. Let $H=(V(\alpha)\cap V(\beta), E(\alpha)\cap E(\beta))$  be the intersection of $\alpha$ and $\beta$.
 Suppose $H$ has $p$ components and $f$ edges.  Since $H$ is the intersection of two different cycles, $H$ must be a forest with $p+f$ vertices. So $\alpha\cup\beta$ has $2j+2k-p-f$ vertices and $2j+2k-f$ edges. 
 Let $a,b$ be the number of vertices in $\alpha\cup\beta$ that are  from $V_1$ and $V_2$, respectively.  Then \begin{align}\label{eq:a+b}
     a+b=2j+2k-p-f.
 \end{align}

  Let $v_1,v_2$ be the number of vertices in $V_1$ and $V_2$ for $H$, respectively. Then we have
$a =j+k-v_1, 
  b =j+k-v_2,$ and  $|a-b|=|v_1-v_2|.$ 
  Note that each component in $H$ is a path. For each path, the difference between the number of vertices from $V_1$  and $V_2$ is at most $1$. This implies 
 \begin{align}\label{eq:a-b}
   |a-b|=|v_1-v_2|\leq p.  
 \end{align}

 From the proof of Corollary 21 in  \cite{dumitriu2016marvcenko}, the number of all possible isomorphism types of   $\alpha\cup \beta$ given $|\alpha|, |\beta|\leq 2r$  and $p,f\leq 2r$ is at most 
 \[\frac{(16r^3)^{p-1}}{((p-1)!)^2}.\] 
 For each  isomorphism type, as a subgraph in $K_{n,m}$, the number of ways to label it is at most 
$ [n]_a[m]_b+[n]_b[m]_a, $
 where the two terms  come from assigning vertices in $V_1,V_2$ in two ways (pick an arbitrary starting vertex, decide whether it is from $V_1$ or $V_2$, then choose labels accordingly).

 From  \eqref{eq:a+b}, \eqref{eq:a-b}, and the assumption that $n\leq m$, we have that when $f$ is even,
 \[ [n]_a[m]_b+[n]_b[m]_a \leq  2n^{j+k-p-f/2}m^{j+k-f/2}=2n^{-p}(nm)^{j+k-f/2}.\]
 And when $f$ is odd,
\[
   [n]_a[m]_b+[n]_b[m]_a \leq 2n^{-p+1}(nm)^{j+k-f/2-1/2}. 
\]

By  \eqref{eq:McKay2} in Lemma \ref{lem:Mckay}, the probability of any realization of the isomorphism type as a subgraph in $G$ is bounded by 
 \[\frac{c_1[(d_1-1)(d_2-1)]^{j+k-f/2}}{(nm)^{j+k-f/2}}.\]
 With all the estimates above, the right-hand side of \eqref{eq:A3} is now bounded by  
\begin{align}
&\sum_{j=2}^r\sum_{1\leq p,f\leq 2r}\frac{(16r^3)^{p-1}}{((p-1)!)^2}\cdot \left( [n]_a[m]_b+[n]_b[m]_a \right)\cdot \frac{c_1[(d_1-1)(d_2-1)]^{j+k-f/2}}{(nm)^{j+k-f/2}} \notag\\
\leq & \sum_{j=2}^r\sum_{1\leq p,f\leq 2r}\frac{(16r^3)^{p-1}}{((p-1)!)^2}\cdot  (2n^{-p} (nm)^{j+k-f/2})
\cdot \frac{c_1[(d_1-1)(d_2-1)]^{j+k-f/2}}{(nm)^{j+k-f/2}} \mathbf{1}\{ \text{$f$ is even}\}\notag\\
&+\sum_{j=2}^r\sum_{1\leq p,f\leq 2r}\frac{(16r^3)^{p-1}}{((p-1)!)^2}\cdot  (2n^{-p+1} (nm)^{j+k-f/2-1/2})
\cdot \frac{c_1[(d_1-1)(d_2-1)]^{j+k-f/2}}{(nm)^{j+k-f/2}} \mathbf{1}\{ \text{$f$ is odd}\}\notag \\
=&O\left( \frac{[(d_1-1)(d_2-1)]^{k+r-1}}{n}\right)+O\left( \frac{[(d_1-1)(d_2-1)]^{r+k-1/2}}{(nm)^{1/2}}\right)\notag \\
=&O\left( \frac{[(d_1-1)(d_2-1)]^{k+r}}{nd_1}\right).\label{eq:sumbound3}
\end{align}

Combining all estimates from \eqref{eq:sumbound1}, \eqref{eq:sumbound2} and \eqref{eq:sumbound3}, we finally obtain  
\begin{align}\label{eq:225}
\sum_{\alpha\in\mathcal J_k}
	\mathbb E\left| I_{\alpha}-\frac{F_{\alpha}}{[n]_{k}[m]_{k}(d_1d_2)^{k}}\right|&=O\left( \frac{[(d_1-1)(d_2-1)]^{r+k}}{nd_1}\right).
\end{align}
This provides an upper bound for \eqref{eq:twosum2}.

(3) \textit{The upper bound on \eqref{eq:twosum}}.
Now the upper bounds on \eqref{eq:twosum1} and \eqref{eq:twosum2} have been provided in  \eqref{eq:Balpha} and  \eqref{eq:225}, respectively. We are ready to estimate \eqref{eq:twosum}. Recall
\[  \xi_k=\min \{ 1,1.4\mu_k^{-1/2}\}=\frac{2.8\sqrt{k}}{[(d_1-1)(d_2-1)]^{k/2}}.\]
Then from \eqref{eq:Balpha} and \eqref{eq:225}, there is an absolute constant $c_7>0$ such that \eqref{eq:twosum} is bounded by 
\begin{align}
\sum_{k=2}^r \frac{c_7\sqrt{k}[(d_1-1)(d_2-1)]^{r+k/2}}{nd_1}=O\left( \frac{\sqrt{r}[(d_1-1)(d_2-1)]^{3r/2}}{nd_1}\right). 
\end{align} This completes the proof.
\end{proof}

\subsection{Cyclically non-backtracking walks and the Chebyshev polynomials}\label{sec:chebypoly}

In this section, we study non-backtracking walks in  biregular bipartite graphs and relate them to the Chebyshev polynomials. The relation will be used in Section \ref{sec:CLT} to study eigenvalue fluctuations for random biregular bipartite graphs.

\begin{definition}[non-backtracking walk]\label{def:NBW}
We define a \textit{non-backtracking walk} of length $2k$ in a biregular bipartite graph to be a walk $(u_1,v_1,\dots, u_{k},v_k,u_{k+1})$ such that $u_i\in V_1, v_i\in V_2$, $u_{i+1}\not=u_i$, for all $1\leq i\leq k$ and $v_{i+1}\not=v_i$ for all $1\leq i\leq k-1$. Note that in our definition, all such walks start  and end at some vertices from $V_1$.
\end{definition}

\begin{figure}[ht]
    \centering
    \includegraphics{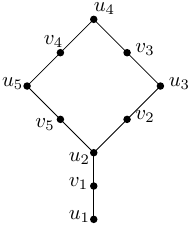}
    \caption{  $(u_2,v_2,u_3,v_3,u_4,v_4,u_5,v_5,u_2)$ is a cyclically non-backtracking walk.  $(u_1,v_1,u_2,v_2,u_3,v_3,u_4,v_4,u_5,v_5,u_2,v_1,u_1)$ is a closed non-backtracking but is not cyclically non-backtracking.}
    \label{fig:NBW}
\end{figure}

\begin{definition}[cyclically non-backtracking walk]
A walk of length $2k$ denoted by \[(u_1,v_1,\dots, u_{k},v_k,u_{k+1})\] is closed if $u_{k+1}=u_1$.
\textit{A cyclically non-backtracking walk} is a closed non-backtracking walk such that its last two steps are not the reverse of its first two steps. Namely, $(u_1,v_1,u_2)\not=(u_{k+1},v_{k},u_k)$.  Figure \ref{fig:NBW} gives an example of a closed non-backtracking walk that is not cyclic non-backtracking.
\end{definition}

Let $G_n$ be a random $(d_1,d_2)$-biregular bipartite graph and $C_k^{(n)}$ be the number of cycles of length $2k$ in $G_n$. Denote $\textnormal{NBW}_k^{(n)}$ to be the number of non-backtracking walk of length $2k$, and $\textnormal{CNBW}_k^{(n)}$ to be the number of cyclically non-backtracking walks of length $2k$ in $G_n$. 
Let $(C_k^{(\infty)},k\geq 2)$  be independent Poisson random variables with mean \[\mu_k=\frac{[(d_1-1)(d_2-1)]^{k}}{2k}.\] We also define $C_1^{(\infty)}=C_1^{(n)}=0$. For $k\geq 1$,
denote 
\begin{align}\label{eq:CNBWinfinity}
    \CNBW_{k}^{(\infty)}=\sum_{j\mid k} 2jC_j^{(\infty)}. 
\end{align}  	

For any cycle of length $2j$ in $G_n$ with $j\mid k$, we can obtain $2j$ cyclically non-backtracking walks by choosing a starting point from $V_1$, fixing a direction and then walking around the cycle of length $2k$ repeatedly. The next lemma shows that $\textnormal{CNBW}_k^{(n)}$ can be approximated by the count of those repeated walks around cycles.

\begin{lemma}\label{lem:EBk}
Let $G_n$ be a random $(d_1,d_2)$-biregular bipartite graph.
	Suppose $d_1\leq n^{1/3},k\leq n^{1/10}$, define
	\begin{align}\label{eq:defBkn}
	B_k^{(n)}=\CNBW_k^{(n)}-\sum_{j\mid k} 2jC_j^{(n)}    
	\end{align} 
to be the number of cyclically non-backtracking walks of length $2k$ in $G_n$ that are not repeated walks around cycles. Then 
\[ \mathbb EB_k^{(n)}\leq \frac{c_7k^7[(d_1-1)(d_2-1)]^{k}}{n}.\]
\end{lemma}

	We call a cyclically non-backtracking walk \textit{bad} if it's not a repeated walk on a cycle. Then from \eqref{eq:defBkn}, $B_k^{(n)}$ counts the number of bad cyclically non-backtracking walks of length $2k$.  
	
	Let $(w_0,w_1,\dots, w_{2k})$ with $w_{2k}=w_0\in V_1$ be a bad cyclically non-backtracking walk in $K_{n,m}$ of length $2k$. For any $1\leq i\leq 2k$, we say that the $i$-th step of the walk is  
	\begin{itemize}
	    \item \textit{free} if $w_i$ did not previously occur in the walk;
	    \item \textit{a coincidence} if $w_i$ previously occurred in the walk, but the edge $w_{i-1}w_i$ didn't;
	    \item \textit{forced} if the edge $w_{i-1}w_i$ previously occurred in the walk.
	\end{itemize}
	
Let $\chi+1$ be the number of coincidences and $f$ be the number of forced steps in the walk. Let $\chi_1+1$ and $\chi_2$ be the number of coincidence steps ending at a vertex from $V_1$ and $V_2$, respectively. Let $f_1,f_2$ be the number of forced steps ending at a vertex from $V_1$ and $V_2$, respectively. Denote $v,e$  the number of distinct vertices and edges in the cyclically non-backtracking walk, respectively.  We now have the following relations:
\begin{align*}
    \chi+1&=\chi_1+\chi_2+1,\\
    f&=f_1+f_2,\\
    v&=(2k+1)-(\chi+1)-f=2k-\chi-f,\\
    e&=2k-f.
\end{align*}
For any repeated walk on a cycle, the number of coincidences is $1$ and $\chi=0$. Therefore if the walk is bad, we must have $\chi\geq 1$.

The following lemma bounds the number of cyclically non-backtracking walks with given parameters $\chi_1,\chi_2,f_1$, and $f_2$.
\begin{lemma}\label{lem:Bk}
    Consider cyclically non-backtracking walks of length $2k$ on $K_{n,m}$ such that in the subgraph spanned by this walk, all vertices from $V_1$ have degrees at most $d_1$ and vertices from $V_2$ have degrees at most $d_2$. Then the number of such walks with given $\chi_1,\chi_2, f_1,f_2$ satisfying $\chi\geq 1$  is at most \[(2k)^{3(\chi_1+\chi_2)+2}(d_1-1)^{f_2}(d_2-1)^{f_1} n^{k-\chi_1-f_1}m^{k-\chi_2-f_2}.\] Moreover, we must have $|f_1-f_2|\leq \chi+1$.
\end{lemma}
\begin{proof} 
We count the number of such cyclically non-backtracking walks by choosing the coincidences, forced steps, and free steps separately. 
Given that there are $\chi+1$ coincidences, there are ${\binom{2k}{\chi+1}}$ many possible subsets of indices in $\{1,\dots,2k\}$ where coincidences can happen. The vertices at a coincidence have already occurred in the walk, so there are at most $2k$ choices for each of them, giving us a total of 
${\binom{2k}{\chi+1}}(2k)^{\chi+1}\leq (2k)^{2\chi+2}$ many choices. 

For forced steps, they can only occur after a coincidence or another forced step. After each coincidence, imagine assigning some number of steps to be forced. The number of ways to do this is at most the number of weak compositions of $f$ elements into $\chi+1$ parts, which is ${\binom{f+\chi}{\chi}}\leq (2k)^{\chi}$. For each forced step ending at a vertex from $V_1$, the walk can only move along an edge that has already been traversed, so there are at most $(d_2-1)$ possible choices 
at every step due to the non-backtracking property. Similarly, for each forced step ending at a vertex from $V_2$ there are at most $d_1-1$ possible choices. Altogether this gives us at most $(2k)^{\chi}(d_1-1)^{f_2}(d_2-1)^{f_1}$ choices for all forced steps.

There are  $k-\chi-1-f_1$ many free steps ending at a vertex from $V_1$, we have at most $n$ choices for the next vertex, and we have an additional $n$ choices for $w_0\in V_1$, which gives a total of at most $n^{k-\chi_1-f_1}$ many choices. Similarly, the number of free steps ending at a vertex from $V_2$ is at most $m^{k-\chi_2-f_2}$. Multiplying together every part from coincidences, forced steps, and free steps give us at most 
\[ (2k)^{3\chi+2}(d_1-1)^{f_2}(d_2-1)^{f_1}n^{k-\chi_1-f_1}m^{k-\chi_2-f_2}\]
many such cyclically non-backtracking walks.

Next, we bound $|f_1-f_2|$. Recall that forced steps can only occur after a coincidence or another forced step. Then there are at most $\chi+1$ many consecutive forced steps starting from a certain coincidence step. In each consecutive forced step, the number of vertices from $V_1$ and $V_2$ differ by at most $1$, since the subgraph spanned by any consecutive forced steps is a path. Hence we have $|f_1-f_2|\leq \chi+1.$
\end{proof}

Equipped with Lemma \ref{lem:Bk}, we continue to prove Lemma \ref{lem:EBk}.
\begin{proof}[Proof of Lemma \ref{lem:EBk}]
By Part (a) in Lemma \ref{lem:Mckay}, the probability that a given bad walk appears in $G_n$ is at most \[c_1\left(\frac{(d_1-1)(d_2-1)}{nm}\right)^{k-f/2}.\] From the upper bound on the number of such walks in Lemma \ref{lem:Bk}, summing over all possibilities of $\chi_1,\chi_2,f_1,f_2$, we  have \begin{align} \footnotesize
    &\mathbb EB_{k}^{(n)} \notag\\
    \leq & \sum_{\substack{\chi_1,\chi_2:\\\chi_1+\chi_2\geq 1}}\sum_{\substack{0\leq f_1,f_2\leq k-1\\ |f_1-f_2|\leq \chi+1}} (2k)^{3\chi+2}(d_1-1)^{f_2}(d_2-1)^{f_1}n^{k-\chi_1-f_1}m^{k-\chi_2-f_2}c_1 \left(\frac{(d_1-1)(d_2-1)}{nm}\right)^{k-f/2} \notag\\
    % &=[(d_1-1)(d_2-1)]^k\sum_{\chi_1+\chi_2\geq 1}n^{-\chi_1}m^{-\chi_2} (2k)^{3(\chi_1+\chi_2)+2}\sum_{f_1,f_2=0}^{k-1}(d_1-1)^{(f_2-f_1)/2}(d_2-1)^{(f_1-f_2)/2}n^{(f_2-f_1)/2}m^{(f_1-f_2)/2}\\
    =&c_1[(d_1-1)(d_2-1)]^k\sum_{\chi_1+\chi_2\geq 1}n^{-\chi_1}m^{-\chi_2} (2k)^{3(\chi_1+\chi_2)+2}\sum_{\substack{0\leq f_1,f_2\leq k-1\\ |f_1-f_2|\leq \chi+1}}\left(\frac{(d_2-1)m}{(d_1-1)n}\right)^{(f_1-f_2)/2}.\notag 
\end{align}

Since  $(d_2-1)d_1\leq (d_1-1)d_2$, the following inequality holds:
\begin{align}\label{eq:dnkkk}
    \sum_{\substack{0\leq f_1,f_2\leq k-1\\ |f_1-f_2|\leq \chi+1}}\left(\frac{(d_2-1)m}{(d_1-1)n}\right)^{(f_1-f_2)/2}=& \sum_{\substack{0\leq f_1,f_2\leq k-1\\ |f_1-f_2|\leq \chi+1}} \left(\frac{(d_2-1)d_1}{(d_1-1)d_2}\right)^{(f_1-f_2)/2}\\
    \leq & k^2\left(\frac{(d_1-1)d_2}{(d_2-1)d_1}\right)^{(\chi+1)/2}.\notag
\end{align}
Since $d_1\leq n^{1/3}, k\leq n^{1/10}$,  \eqref{eq:dnkkk} implies 
\begin{align*}
  \mathbb EB_{k}^{(n)}&\leq c_1 k^2[(d_1-1)(d_2-1)]^k\sum_{\chi_1+\chi_2\geq 1}n^{-\chi_1}m^{-\chi_2} (2k)^{3(\chi_1+\chi_2)+2}\left(\frac{(d_1-1)d_2}{(d_2-1)d_1}\right)^{(\chi+1)/2}\\
  &= k^2[(d_1-1)(d_2-1)]^k O\left( \frac{(2k)^5(d_1-1)d_2}{n(d_2-1)d_1}\right)\\
  &=O\left( \frac{k^7(d_1-1)^{k}(d_2-1)^{k}}{n}\right).  
\end{align*}
This completes the proof of Lemma \ref{lem:EBk}.
\end{proof}

Recall the definition of $\CNBW_k^{(\infty)}$ from \eqref{eq:CNBWinfinity}. The following corollary holds.

\begin{cor} \label{cor:dTVCNBW}
Suppose $d_1\leq n^{1/3}$ and $r\leq n^{1/10}$. There exists a constant $c_8>0$ such that
\begin{align}
	d_{\textnormal{TV}}\left( (\CNBW_k^{(n)},2\leq k\leq r),(\CNBW_k^{(\infty)},2\leq k\leq r)\right)\leq \frac{c_8\sqrt{r}[(d_1-1)(d_2-1)]^{3r/2}}{nd_1}. 
\end{align}
\end{cor}

\begin{proof}
	By the definition of total variation distance,  for any measurable map $f$ and random variable $X,Y$, we have 
	\begin{align}\label{eq:TVinequality}
d_{\textnormal{TV}}(f(X),f(Y))\leq d_{\textnormal{TV}}(X,Y).	    
	\end{align}
It follows from Theorem  \ref{thm:Poissoncount} that 
	\begin{align}\label{eq:53}
	d_{\textnormal{TV}}\left( \left(  \sum_{j\mid k} 2jC_j^{(n)},2\leq k\leq r\right),\left(\CNBW_k^{(\infty )},2\leq k\leq r \right)\right)\leq \frac{c_6\sqrt{r}[(d_1-1)(d_2-1)]^{3r/2}}{nd_1}. 
\end{align}
By Markov's inequality and Lemma \ref{lem:EBk},
\begin{align}
    \mathbb P(B_{k}^{(n)}\geq 1)\leq \frac{c_7k^7[(d_1-1)(d_2-1)]^{k}}{n}.
\end{align} 
Summing these probabilities for $k=2,\dots,r$ implies
\begin{align}
    \left(  \sum_{j\mid k} 2jC_j^{(n)},2\leq k\leq r\right)=(\CNBW_k^{(n)}, 2\leq k\leq r)
\end{align}
with probability $1-O\left(\frac{r^7[(d_1-1)(d_2-1)]^r}{n}\right)$.  Therefore by the coupling inequality,
\begin{align}\label{eq:57}
    d_{\textnormal{TV}}\left(\left(  \sum_{j\mid k} 2jC_j^{(n)},2\leq k\leq r\right),(\CNBW_k^{(n)}, 2\leq k\leq r)\right)=O\left(\frac{r^7[(d_1-1)(d_2-1)]^r}{n}\right).
\end{align}
From \eqref{eq:53} and \eqref{eq:57}, 
\begin{align*}
   &d_{\textnormal{TV}}\left( (\CNBW_k^{(n)},2\leq k\leq r),(\CNBW_{k}^{(\infty)},2\leq k\leq r)\right)\\
   \leq & \frac{c_6\sqrt{r}[(d_1-1)(d_2-1)]^{3r/2}}{nd_1}+ O\left(\frac{r^7[(d_1-1)(d_2-1)]^r}{n}\right)=O\left(\frac{\sqrt{r}[(d_1-1)(d_2-1)]^{3r/2}}{nd_1}\right).
\end{align*}
\end{proof}

Let $\lambda_1\geq \cdots\geq \lambda_n$ be the eigenvalues of $\frac{XX^{\top}-d_1I}{\sqrt{(d_1-1)(d_2-1)}}$. For the rest of this section, we  connect the spectrum of $\frac{XX^{\top}-d_1I}{\sqrt{(d_1-1)(d_2-1)}}$ with Chebyshev polynomials and cyclically non-backtracking walks.
Define 
\begin{align}
\Gamma_{0}(x)&=1, \notag\\
\Gamma_{2k}(x)&=2T_{2k}\left(\frac x 2\right)+\frac{d_1-2}{(d_1-1)^k},\label{eq:gamma9}	\\
\Gamma_{2k+1}(x)&=2T_{2k+1}\left(\frac x 2\right).\label{eq:gamma99}
\end{align}
Here $\{T_k(x)\}$ are  the  Chebyshev polynomials of the first kind on $[-1,1]$ which satisfy
\begin{align}
    T_0(x)&=1, \quad 
    T_1(x) =x, \notag\\
    T_{k+1}(x) &=2xT_k(x)-T_{k-1}(x).\label{eq:cheby1}
\end{align}
Let $\{U_k(x)\}$ be the Chebyshev polynomials of the second kind on $[-1,1]$ such that
\begin{align*}
U_{-1}(x) &=0,\quad U_0(x) = 1, \\
U_{k+1}(x) &=2xU_k(x)-U_{k-1}(x).	
\end{align*}
Define
\begin{align}\label{eq:pk}
p_k(x)=U_k\left(\frac{x}{2}\right)-\frac{1}{d_1-1}U_{k-2}\left(\frac x 2\right).	
\end{align}
We begin with representing closed non-backtracking walks with $p_k(x)$. The following lemma gives a deterministic identity. Recall in our Definition \ref{def:NBW}, all closed non-backtracking walks start and end at vertices in $V_1$.
\begin{lemma}
Let $\textnormal{NBW}_k^{(n)}$ be the number of closed non-backtracking walks of length $2k$ in a $(d_1,d_2)$-biregular bipartite graph $G$.  Let $\lambda_1\geq \cdots \geq \lambda_n$ be the eigenvalues of $\frac{XX^{\top}-d_1I}{\sqrt{(d_1-1)(d_2-1)}}$.
We have
\begin{align} \label{eq:plambda}
\sum_{i=1}^n p_k(\lambda_i)=(d_1-1)^{-k/2}(d_2-1)^{-k/2}\textnormal{NBW}_{k}^{(n)}.
\end{align}
\end{lemma}
\begin{proof}
Let $A^{(k)}$ be the $n\times n$ matrix such that $A_{ij}^{(k)}$ is the number of non-backtracking walks of length $2k$ from $i$ to $j$, where $i,j\in V_1$. We have the following relations:
\begin{align}
    A^{(1)} &=XX^{\top}-d_1I, \quad A^{(2)}=(A^{(1)})^2-d_1(d_2-1)I, \notag\\
    A^{(k+1)}&=A^{(1)} A^{(k)}-(d_1-1)(d_2-1)A^{(k-1)}, \quad \forall k\geq 2.\label{eq:Ak2}
\end{align}

The expressions of $A^{(1)}$ and $A^{(2)}$ follow from the definition of non-backtracking walks. Since a non-backtracking walk of length $2k+2$ can be decomposed as a non-backtracking walk of length $2k$ and a non-backtracking walk of length $2$ which avoid backtracking at the $2k$-th step, the expression \eqref{eq:Ak2} holds. 
We now claim that for $k\geq 1$,
\begin{align}\label{eq:pkk}
    p_k\left(\frac{XX^{\top}-d_1I}{\sqrt{(d_1-1)(d_2-1)}}\right)=[(d_1-1)(d_2-1)]^{-k/2}A^{(k)},
\end{align}
and prove it by induction. Note that from \eqref{eq:pk},
\[p_1(x)=x,\quad  p_2(x)=x^2-1-\frac{1}{d_1-1}. \]
It is easy to check \eqref{eq:pkk} holds for $k=1,2$. 
Since $p_k(x)$ is a linear combination of $U_k(x/2)$ and $U_{k-2}(x/2)$, it satisfies the recursive relation for $U_k(x/2)$, which is 
\begin{align*}
    p_k(x)=xp_k(x)-p_{k-1}(x).
\end{align*}

Assume \eqref{eq:pkk} holds for $k\leq s$. Let $M=XX^{\top}-d_1I$. Then
\begin{align*}
    p_{s+1}\left(\frac{M}{\sqrt{(d_1-1)(d_2-1)}}\right)
    =&{M}[(d_1-1)(d_2-1)]^{-(s+1)/2}A^{(s)}-[(d_1-1)(d_2-1)]^{-(s-1)/2}A^{(s-1)}\\
    =&[(d_1-1)(d_2-1)]^{-(s+1)/2}\left(MA^{(s)}-(d_1-1)(d_2-1)A^{(s-1)} \right)\\
    =&[(d_1-1)(d_2-1)]^{-(s+1)/2} A^{(s+1)},
\end{align*}
where the last equality is from \eqref{eq:Ak2}. Therefore \eqref{eq:pkk} holds. Taking  trace on both sides in \eqref{eq:pkk}, we obtain \eqref{eq:plambda}.
\end{proof}

The next theorem  is an algebraic relation  between $\Gamma_k$ and the number of cyclic non-backtracking walks. Together with Lemma \ref{cor:dTVCNBW}, it implies the  polynomials $\Gamma_k(x)$ of the eigenvalues for RBBGs converges in distribution to a sum of Poisson random variables.
\begin{theorem}\label{thm:chebyshevCNBW}
Let $G$ be a $(d_1,d_2)$-biregular bipartite graph and $\lambda_1\geq \cdots \geq \lambda_n$ be the eigenvalues of $\frac{XX^{\top}-d_1I}{\sqrt{(d_1-1)(d_2-1)}}$. Then  for any $k\geq 1$, we have
	\begin{align}\label{eq:formulaCNBW}
\sum_{i=1}^n \Gamma_{k}(\lambda_i)=(d_1-1)^{-k/2}(d_2-1)^{-k/2}	\textnormal{CNBW}_{k}^{(n)}.
\end{align}
\end{theorem}

\begin{proof}

We first relate the number of cyclically non-backtracking closed walks $\CNBW_k^{(n)}$ to the number of  closed non-backtracking walks  $\textnormal{NBW}_k^{(n)}$.

A closed non-backtracking walk of length $2k$ is either cyclically non-backtracking, or it can be obtained from a closed non-backtracking walk of length $2(k-2)$ by adding a new walk of length $2$ (which we call a \textit{tail}) to the beginning of the walk and its reverse to the end (see Figure \ref{fig:NBW} for an example). For any cyclically non-backtracking walk of length $2(k-2)$, we can add a tail in $(d_1-2)(d_2-1)$ many ways. For any closed non-backtracking walk of length $2(k-2)$ that is not cyclically non-backtracking, we can add a tail in $(d_1-1)(d_2-1)$ many ways. 
Therefore for $k\geq 3$, we have the following equation
\begin{align*}
\text{NBW}_k^{(n)}&=\text{CNBW}_k^{(n)}	+(d_1-2)(d_2-1)\text{CNBW}_{k-2}^{(n)}+(d_1-1)(d_2-1)(\text{NBW}_{k-2}^{(n)}-\text{CNBW}_{k-2}^{(n)})\\
&=\text{CNBW}_k^{(n)}+(d_1-1)(d_2-1)\textnormal{NBW}_{k-2}^{(n)}-(d_2-1)\text{CNBW}_{k-2}^{(n)},
\end{align*}
which can be written as 
\begin{align}\label{eq:recursive}
\text{CNBW}_{k}^{(n)}-(d_2-1)\text{CNBW}_{k-2}^{(n)}=\text{NBW}_k^{(n)}-(d_1-1)(d_2-1)\text{NBW}_{k-2}^{(n)}.	
\end{align}

Note that $\textnormal{CNBW}_k^{(n)}=\textnormal{NBW}_k^{(n)}$ for $k=1,2$. Applying \eqref{eq:recursive} recursively,  we have when $k$ is even,
\begin{align} \label{eq:keven}
&\text{CNBW}_{k}^{(n)}\\
=&\text{NBW}_k^{(n)}-(d_1-2) [(d_2-1)\text{NBW}_{k-2}^{(n)}	+(d_2-1)^2\text{NBW}_{k-4}^{(n)}	+\cdots +(d_2-1)^{\frac{k-2}{2}}\text{NBW}_{2}^{(n)})].\notag 
\end{align}
And when $k$ is odd,
\begin{align}\label{eq:kodd}
&\text{CNBW}_{k}^{(n)} \\
=&\text{NBW}_k^{(n)}-(d_1-2)[(d_2-1)\text{NBW}_{k-2}^{(n)}	+(d_2-1)^2\text{NBW}_{k-4}^{(n)}	+\cdots +(d_2-1)^{\frac{k-3}{2}}\text{NBW}_{3}^{(n)}	].\notag 
\end{align}

Denote 
\begin{align*}
   \overline{\text{NBW}}_k^{(n)}:=&(d_2-1)^{-k/2} \text{NBW}_{k}^{(n)}  , \quad  \overline{\text{CNBW}}_k^{(n)}:=(d_2-1)^{-k/2} \text{CNBW}_{k}^{(n)}. 
\end{align*}
 We can simplify the above equations \eqref{eq:keven} and \eqref{eq:kodd} as
\begin{align}\label{eq:barrecursion}
\overline{\text{CNBW}}_{k}^{(n)}=\overline{\text{NBW}}_k^{(n)}-(d_1-2)\left( \overline{\text{NBW}}_{k-2}^{(n)}+\overline{\text{NBW}}_{k-4}^{(n)}+\cdots+\overline{\text{NBW}}_a^{(n)}\right),
\end{align}
where $a=2$ if $k$ is even and $a=1$ if $k$ is odd. Also \eqref{eq:plambda} can be written as 
\begin{align}\label{eq:pklambda}
	\sum_{i=1}^n p_k(\lambda_i)=(d_1-1)^{-k/2}\overline{\text{NBW}}_{k}^{(n)}.
\end{align}

From the proof of  Proposition 32 in \cite{dumitriu2013functional}, we have the following relation between $\Gamma_k(x)$ and $p_k(x)$ for $k\geq 1$:
\begin{align}
\Gamma_{2k}(x)&=p_{2k}(x)-(d_1-2)\left( \frac{p_{2k-2}(x)}{d_1-1}+\frac{p_{2k-4}(x)}{(d_1-1)^2}+\cdots+\frac{p_{2}(x)}{(d_1-1)^{k-1}}\right), \label{eq:gamma2k}\\
\Gamma_{2k-1}(x)&=p_{2k-1}(x)-(d_1-2)\left( \frac{p_{2k-3}(x)}{d_1-1}+\frac{p_{2k-5}(x)}{(d_1-1)^2}+\cdots+\frac{p_{1}(x)}{(d_1-1)^{k-1}}\right).\label{eq:gamma2k1}
\end{align}
Then from \eqref{eq:pklambda} and \eqref{eq:barrecursion},
\begin{align*}
 (d_1-1)^{-k} \overline{\text{CNBW}}_{2k}^{(n)}&=\sum_{i=1}^n\left(  p_{2k}(\lambda_i)-(d_1-2) \left( \frac{p_{2k-2}(\lambda_i)}{d_1-1}+\cdots+\frac{p_{2}(\lambda_i)}{(d_1-1)^{k-1}}\right)\right)=\sum_{i=1}^n \Gamma_{2k}(\lambda_i),
\end{align*}
where the last equality is from \eqref{eq:gamma2k}. Similarly, from \eqref{eq:gamma2k1},
\begin{align*}
    (d_1-1)^{(2k-1)/2}\overline{\text{CNBW}}_{2k-1}^{(n)}=\sum_{i=1}^n \Gamma_{2k-1}(\lambda_i).
\end{align*}

Therefore for all $k\geq 1$,
\begin{align*}
\sum_{i=1}^n \Gamma_{k}(\lambda_i)=(d_1-1)^{-k/2}	\overline{\text{CNBW}}_{k}^{(n)}=[(d_1-1)(d_2-1)]^{-k/2}	\text{CNBW}_{k}^{(n)}.
\end{align*}
This completes the proof of Theorem \ref{thm:chebyshevCNBW}.
\end{proof}

\section{Spectral gap}\label{sec:spectralgap}

In this section, we  provide some estimates on the second largest eigenvalue of the random biregular bipartite graphs that will be used to study eigenvalue fluctuations in Section \ref{sec:CLT}.
Note that the largest eigenvalue of $XX^{\top}-d_1I$ is $\lambda_1=d_1(d_2-1)$. In the next theorem, we provide upper bounds on $|\lambda|$ for all eigenvalues $\lambda\not=\lambda_1$.
\begin{theorem}\label{thm:spectralgap}
Let	$G$ be a $(d_1,d_2)$-random biregular bipartite graph with $d_1\geq d_2$. Let $\lambda_1\geq \cdots\geq \lambda_n$ be the eigenvalues of $XX^{\top}-d_1I$. 
\begin{enumerate}
	\item For fixed $d_1,d_2$, there exists a sequence $\epsilon_n\to 0$ such that for any eigenvalue $\lambda \not=\lambda_1$,
	 \begin{align}
	     \mathbb P( |\lambda-(d_2-2)|\geq 2\sqrt{(d_1-1)(d_2-1)}+\epsilon_n)\to 0 \label{eq:gapspectral1}
	 \end{align} 
	as $n\to\infty$.
	\item Suppose $d_2\leq \frac{1}{2}n^{2/3}$, $d_1\geq d_2\geq cd_1$ for some constant $c\in (0,1)$. Then for some constant $\alpha_1>0$ depending on $c$ and any eigenvalue $\lambda \not=\lambda_1$,
	\begin{align}\label{eq:E1}
	\mathbb P\left( |\lambda |\geq \alpha_1\sqrt{(d_1-1)(d_2-1)}\right)\leq  \frac{1}{n^2}.    
	\end{align} 
\item Suppose $d_2\leq C_1$, $d_1\leq n^2$, there exists a constant $\alpha_2$ depending on $C_1$ such that for any eigenvalue $\lambda \not=\lambda_1$,
	\begin{align}\label{eq:E2}
\mathbb P\left( |\lambda|\geq \alpha_2\sqrt{(d_1-1)(d_2-1)}\right)\leq  \frac{1}{n^2}.	    
	\end{align} 
\end{enumerate}
\end{theorem}
\begin{remark}
The probability estimates in \eqref{eq:E1} and \eqref{eq:E2} can be improved, see \cite{zhu2020second}. In order to prove the main theorems in Section \ref{sec:CLT}, we only include a weaker version for simplicity.
\end{remark}

\begin{proof}[proof of Theorem \ref{thm:spectralgap}]
Theorem 4 in \cite{brito2018spectral} states that for a random biregular bipartite graph with $d_1\geq d_2$, the eigenvalues of the adjacency matrix $A$ satisfy the following estimates with high probability:
\begin{enumerate}
    \item the second eigenvalue of $A$ satisfies $\lambda_2(A)\leq \sqrt{d_1-1}+\sqrt{d_2-1}+o(1)$,
    \item the smallest positive eigenvalue of $A$ satisfies $\lambda_{\min}^+(A)\geq \sqrt{d_1-1}-\sqrt{d_2-1}-o(1).$
\end{enumerate}
Since eigenvalues of $XX^{\top}$ are the squares of the eigenvalues for $A$, we have with high probability,
\begin{align*}
    \lambda_2(XX^{\top})-d_1-(d_2-2) &\leq 2\sqrt{(d_1-1)(d_2-1)}+o(1),\\
    \lambda_{n}(XX^{\top})-d_1-(d_2-2) &\geq -2\sqrt{(d_1-1)(d_2-1)}-o(1),
\end{align*}
therefore
 \eqref{eq:gapspectral1} holds.

Theorem 1.1 in \cite{zhu2020second}
states that if $d_2\leq \frac{1}{2}{n^{2/3}}$ and $d_1\geq d_2$, there exists a constant $\alpha>0$ such that $\lambda_2(A)\leq \alpha\sqrt{d_1}$ with probability at least $1-m^{-2}$.  This implies for any eigenvalue $\lambda$ of $XX^{\top}-d_1I$ with $\lambda\not=d_1(d_2-1)$, we have 
\begin{align*}
   \mathbb P\left( -d_1\leq \lambda\leq \alpha^2 d_1-d_1\right)\geq 1-m^{-2}\geq 1-n^{-2}.
\end{align*}
Since $d_1\geq d_2\geq cd_1$, we can find a constant $\alpha_1>0$ depending on $\alpha$ and $c$ such that 
\begin{align*}
   \mathbb P\left( |\lambda|\leq \alpha_1\sqrt{(d_1-1)(d_2-1)} \right) \geq 1-n^{-2}.
\end{align*}
Therefore \eqref{eq:E1} holds. Theorem 1.5 in \cite{zhu2020second} states that if $d_2\leq C_1, d_1\leq n^2$, there exists a constant $\alpha_2$ depending on $C_1$ such that
\begin{align*}
    \mathbb P \left(\max_{2\leq i\leq m+n-1}|\lambda_i^2(A)-d_1|\geq \alpha_2 \sqrt{(d_1-1)(d_2-1)}  \right)\leq n^{-2}.
\end{align*}
Then \eqref{eq:E2} follows from the algebraic relation between the spectra of $A$ and $XX^{\top}-d_1I$.
\end{proof}

\section{Eigenvalue fluctuations}\label{sec:CLT}

Lemma \ref{thm:chebyshevCNBW} and Corollary \ref{cor:dTVCNBW} in Section \ref{sec:chebypoly} imply the limiting laws for $\sum_{i=1}^n \Gamma_k(\lambda_i)$ are given by a sum of Poisson random variables. In this section, we extend the results to a more general class of function $f$ and study the behavior of $\sum_{i=1}^n f(\lambda_i)$ for RBBGs with fixed and growing degrees.

The following set-up for weak convergence will be used in Section \ref{sec:gaussian} to prove Theorem \ref{thm:GaussianCLT}. We will closely follow the definitions and notations used in \cite{dumitriu2013functional}. See Section 2 in \cite{dumitriu2013functional} for more details.

Denote $\mathbb N:=\{1,2,\dots\}$.
Let $\vec{w}=(w_m)_{m\in\mathbb N}$ be a sequence of positive weights. Let $L^2(\vec{w})$ be the space of sequences $(x_m)_{m\in\mathbb N}$ that are square-integrable with respect to $\vec{w}$, i.e., $\sum_{m=1}^{\infty} x_m^2w_m<\infty$. We define a complete separable metric space $\mathcal X=(L^2(\vec w), \| \cdot \|)$, where for any  sequence $(x_m)_{m\in \mathbb N}$, \[\|x\|=\left(\sum_{m=1}^{\infty} x_m^2w_m\right)^{1/2}.\]

Denote the space of probability measures on the Borel $\sigma$-algebra of $\mathcal X$ by $\mathcal P(\mathcal X)$. We use the Prokhorov metric for weak convergence as the metric on $\mathcal P(\mathcal X)$. The following results are proved in Section 2 of \cite{dumitriu2013functional}.
\begin{prop}[Lemma 2-4 in \cite{dumitriu2013functional}] \label{prop:weakconvergence} The following holds for the complete separable metric space $\mathcal X$.
    \begin{enumerate}
    \item Let $(a_m)_{m\in\mathbb N}\in L^2(\vec{w})$ be such that $a_m\geq 0$ for every $m$. Then the set
    \[ \{ (b_m)_{m\in\mathbb N}\in L^2(\vec{w}): 0\leq |b_m|\leq a_m, \forall m\in \mathbb N\}\]
    is compact in $(L^2(\vec{w}), \|\cdot \|)$.
        \item Suppose $\{X_n\}$ and $X$ are random sequences taking values in $L^2(\vec{w})$ such that $X_n$ converges in distribution to $X$. Then for any $b\in L^2(\vec{w})$, the random variables $\langle b, X_n\rangle $ converges in distribution to $\langle b, X\rangle$.
        \item Let $x\in \mathcal X$ and $P,Q$ be two probability measures in $\mathcal P(\mathcal X)$. Suppose for any finite collection of indices $(i_1,\dots,i_k)$, the law of random vector $(x_{i_1},\dots, x_{i_k})$ is the same under both $P$ and $Q$. Then $P=Q$ on the entire Borel $\sigma$-algebra of $\mathcal X$.
    \end{enumerate}
\end{prop}

We  also need the following results from the approximation theory.
\begin{definition}[Bernstein ellipse]
For $\rho>1$, let $E_{\rho}$ be the image of the map $z\mapsto (z+z^{-1})/2$ of the open disc of radius $\rho$ in the complex plain centered at the origin. We can $E_{\rho}$ the \textit{Bernstein ellipse} of radius $\rho$.
The ellipse has foci at $\pm 1$, and the sum of the major semi-axis and minor semi-axis is exactly $\rho$.
\end{definition}

\begin{prop}[\cite{trefethen2013approximation}, Theorem 8.1]\label{prop:approx} 
    Suppose $f: [-1,1]\to \mathbb R$ can be analytically extended to $E_{\rho}$ and is bounded by $M$ on $E_{\rho}$. Then $f$ has a unique expansion on $[-1,1]$ as
    \[ f(x)=\sum_{k=0}^{\infty} a_kT_k(x),\]
    where $T_k(x)$ is the Chebyshev polynomial of the first kind defined in  \eqref{eq:cheby1},
    and the coefficients of this expansion satisfy
    \[ |a_0|\leq M, \quad |a_k|\leq \frac{2M}{\rho^k}.\]
\end{prop}
Define $f_k(x)=\sum_{i=0}^k a_kT_k(x)$.
Applying the bound $|T_{k}(x)|\leq 1$ when $x\in [-1,1]$  and Proposition \ref{prop:approx}, we obtain for all $x\in [-1,1]$,
\begin{align}\label{eq:fkapprox}
    |f(x)-f_k(x)|\leq \frac{2M}{\rho^k(\rho-1)}.
\end{align}
\subsection{Poisson fluctuations with fixed degrees}
Now fix $d_1$ and $d_2$ as constants. We are ready to extend our results in Section \ref{sec:chebypoly} to a more general class of functions as follows. Note that the following theorem is given for a sequence of RBBGs with growing $n$. For ease of notation, we drop the dependence on $n$ when writing the matrix $X$ and eigenvalues $\lambda_1,\dots,\lambda_n$.

\begin{theorem}\label{thm:CLTfixed}
For fixed $d_1\geq d_2\geq 2$ and $(d_1,d_2)\not=(2,2)$, let $G_n$ be a sequence of  random $(d_1,d_2)$-biregular bipartite graph. Let $\lambda_1\geq \cdots\geq \lambda_n$ be the eigenvalues of $\frac{XX^{\top}-d_1I}{\sqrt{(d_1-1)(d_2-1)}}$. Suppose $f$ is a function such that $f(2z)$ is analytic on $E_{\rho}$, where $\rho=[ (d_1-1)(d_2-1)]^{\alpha}$ for some $\alpha>\frac{7}{2}$. Then $f(x)$ can be expanded on $[-2,2]$ as
\begin{align}\label{eq:expandf1}
    f(x)=\sum_{k=0}^{\infty} a_k\Gamma_k(x),
\end{align}
and the random variable
\begin{align}\label{eq:YFn}
    Y_{f}^{(n)}:=\sum_{i=1}^n f(\lambda_i)-n a_0
\end{align} 
converges in distribution as $n\to\infty$ to the infinitely divisible random variable
\begin{align}\label{eq:yf}
   Y_{f}:=\sum_{k=2}^{\infty} \frac{a_k}{[(d_1-1)(d_2-1)]^{k/2}}\CNBW_k^{(\infty)}, 
\end{align} 
where $\CNBW_k^{(\infty)}$ is defined in \eqref{eq:CNBWinfinity}.
\end{theorem}

\begin{proof}
Define
\[ f_k(x):=\sum_{i=0}^k a_i\Gamma_i(x).\]
We first show that $f_k(x)$ is a good approximation of $f(x)$. Applying Proposition \ref{prop:approx} to $f(2x)$ gives an expansion \eqref{eq:expandf1} with 
\begin{align}\label{eq:ak}
  |a_k|\leq C[(d_1-1)(d_2-1)]^{-\alpha k}  
\end{align} 
for some constant $C$ that depends only on $d_1,d_2$ and the constant $M$ given in Proposition \ref{prop:approx}.

By the proprieties of Chebyshev polynomials, on any interval $[-K,K]$, we have 
\begin{align}\label{eq:maxT}
    \max_{|x|\leq K}|T_k(x)|=\frac{(K-\sqrt{K^2-1})^k+(K+\sqrt{K^2-1})^k}{2}.
\end{align}
From \eqref{eq:gamma9} and \eqref{eq:gamma99}, we have $\Gamma_1(x)=x$, and for any $k\geq 2$,
\begin{align}\label{eq:gggggk}
   | \Gamma_k(x)|\leq 2 \left| T_k\left(\frac{x}{2}\right)\right|+\frac{d_1-2}{(d_1-1)^{k/2}}\leq 2\left| T_k\left(\frac{x}{2}\right)\right|+1.
\end{align}
From \eqref{eq:maxT},
\begin{align*}
    \max_{|x|\leq 3} 2\left|T_k\left(\frac{x}{2}\right)\right|=\left(\frac{3}{2}-\frac{\sqrt{5}}{2}\right)^k+\left(\frac{3}{2}+\frac{\sqrt{5}}{2}\right)^k,
\end{align*}
Then for all  $k\geq 2$, with \eqref{eq:gggggk} we obtain
\begin{align}\label{eq:3k}
   \sup_{|x|\leq 3}| \Gamma_k(x)|\leq \left(\frac{3}{2}+\frac{\sqrt{5}}{2}\right)^k+2\leq  3^{k+1},
\end{align}
and the same bound holds when $k=1$. From  \eqref{eq:ak} and  \eqref{eq:3k}, for all  $x\in [-3,3]$,
\begin{align*}
    \sum_{k=0}^{\infty}|a_k\Gamma_k(x)|\leq 3C\sum_{k=0}^{\infty} [3((d_1-1)(d_2-1))^{-\alpha}]^{k}<\infty,
\end{align*}
where the last inequality comes from the fact that $(d_1-1)(d_2-1)\geq 2$ and $\alpha>\frac{7}{2}$. Hence the series $\sum_{k=0}^{\infty} a_k\Gamma_k(x)$ is absolutely convergent on $[-3,3]$, which implies the expansion of $f$ in \eqref{eq:expandf1} is valid on $[-3,3]$.
Then we have for a constant $C_1>0$ depending on $C$,
\begin{align}\label{eq:33estimate}
    \sup_{|x|\leq 3}|f(x)-f_k(x)| &\leq  \sup_{|x|\leq 3} \sum_{i=k+1}^{\infty}|a_i\Gamma_i(x)|\leq C_1[3(d_1-1)(d_2-1)^{-\alpha}]^{k+1}. 
\end{align} 
 Denote \[K_1:=\lambda_1=\frac{d_1\sqrt{(d_2-1)}}{\sqrt{d_1-1}}.\] For sufficiently large $k$, from \eqref{eq:maxT} and \eqref{eq:gggggk},
\begin{align*}
     \sup_{|x|\leq K_1}| \Gamma_k(x)|\leq (2K_1)^k. 
\end{align*}
And from \eqref{eq:ak} and the assumption $\alpha>7/2$,
\begin{align*}
   \sum_{k=0}^{\infty}|a_k\Gamma_k(x)|\leq C\sum_{k=0}^{\infty} [2K_1((d_1-1)(d_2-1))^{-\alpha}]^{k}<\infty. 
\end{align*}
It implies the series $\sum_{k=0}^{\infty} a_k\Gamma_k(x)$ is also absolutely convergent on $[-K_1,K_1]$, and the expansion of $f$ in \eqref{eq:expandf1} is  valid on $[-K_1,K_1]$.

Since $2K_1\leq 4[(d_1-1)(d_2-1)]^{1/2}$ and $(d_1-1)(d_2-1)\geq 2$, for a constant $C_2>0$,
\begin{align}\label{eq:unifK}
    \sup_{|x|\leq K_1}|f(x)-f_k(x)|&\leq C_2[2K_1((d_1-1)(d_2-1))^{-\alpha}]^{k+1} \notag\\
    &\leq C_2\left[4((d_1-1)(d_2-1))^{-\alpha+\frac{1}{2}}\right]^{k+1} \leq C_2\left[((d_1-1)(d_2-1))^{-\alpha+\frac{5}{2}}\right]^{k+1}.
\end{align} 
Therefore $f_k$ converges to $f$ uniformly on $[-K_1,K_1]$, and the interval $[-K_1,K_1]$ deterministically contains all eigenvalues of $\frac{XX^{\top}-d_1I}{\sqrt{(d_1-1)(d_2-1)}}$.

By the definition of $\CNBW_k^{(\infty)}$ in \eqref{eq:CNBWinfinity}, Equation \eqref{eq:yf} can be written as 
\[ Y_{f}:=\sum_{j=1}^{\infty}\sum_{i=1}^{\infty} \frac{a_{ij}}{[(d_1-1)(d_2-1)]^{ij/2}}2jC_j^{(\infty)},\]
where $Y_f$ is a sum of independent random variables, and
$\mathbb E|Y_f|^2<\infty$ by \eqref{eq:ak}. 

Denote $\alpha':=\alpha-2>\frac{3}{2}$. 
Choose $\beta$ satisfying $\frac{1}{\alpha'}<\beta<\frac{2}{3}$ and define
\begin{align*}
    r_n &=\left\lfloor\frac{\beta \log n}{\log[(d_1-1)(d_2-1)]} \right\rfloor,\\
    X_f^{(n)}&=\sum_{k=1}^{r_n}\frac{a_k}{[(d_1-1)(d_2-1)]^{k/2}}\CNBW_k^{(n)},\\
     \tilde{Y}_f^{(n)}&=\sum_{k=1}^{r_n} \frac{a_k}{[(d_1-1)(d_2-1)]^{k/2}}\CNBW_k^{(\infty)}.
\end{align*}
Note that $\CNBW_1^{(n)}=0$, from \eqref{eq:formulaCNBW},
\begin{align}\label{eq:XFn}
     X_f^{(n)}&=\sum_{k=2}^{r_n}\frac{a_k}{[(d_1-1)(d_2-1)]^{k/2}}\CNBW_k^{(n)}=\sum_{i=1}^n f_{r_n}(\lambda_i)-na_0.
\end{align}
By Corollary \ref{cor:dTVCNBW}, 

\begin{align*}
	d_{\textnormal{TV}}\left( (\CNBW_k^{(n)},2\leq k\leq r_n),(\CNBW_k^{(\infty)},2\leq k\leq r_n)\right)\leq \frac{c_8\sqrt{r_n}[(d_1-1)(d_2-1)]^{3r_n/2}}{nd_1}=o(1). 
\end{align*}

Since $X_f^{(n)}$ and $\tilde{Y}_f^{(n)}$ are measurable functions of $$(\CNBW_k^{(n)},2\leq k\leq r_n) \quad \text{and} \quad  (\CNBW_k^{(\infty)},2\leq k\leq r_n),$$ respectively, we have
\begin{align*}
    d_{\textnormal{TV}}\left(X_f^{(n)},\tilde{Y}_f^{(n)}\right)\leq d_{\textnormal{TV}}\left( (\CNBW_k^{(n)},2\leq k\leq r_n),(\CNBW_k^{(\infty)},2\leq k\leq r_n)\right)=o(1).
\end{align*}
Note that $\tilde{Y}_f^{(n)}$ converges  almost surely to $Y_f$ by \eqref{eq:ak}, so $X_f^{(n)}$ converges in distribution to $Y_f$.  

By Slutsky's theorem, to show $Y_f^{(n)}$ defined in \eqref{eq:YFn} converges in distribution to $Y_f$, it remains to  show that $Y_f^{(n)}-X_f^{(n)}$ converges to zero in probability.  The largest eigenvalue of $\frac{XX^{\top}-d_1I}{\sqrt{(d_1-1)(d_2-1)}}$ is $K_1$, so from \eqref{eq:unifK} we have \[ \lim_{k\to\infty}f_k(\lambda_1)= f(\lambda_1).\] 
Then for any $\delta>0$ and sufficiently large $n$,
\begin{align}\label{eq:delta2}
  |f(\lambda_1)-f_{r_n}(\lambda_1)|\leq \delta/2.  
\end{align}
 From \eqref{eq:YFn}, \eqref{eq:XFn} and \eqref{eq:delta2}, we  have for sufficiently large $n$,
\begin{align}\label{eq:largen}
    \left|Y_f^{(n)}-X_f^{(n)}\right|&\leq \sum_{i=1}^n |f(\lambda_i)-f_{r_n}(\lambda_i)|\leq \frac{\delta}{2}+\sum_{i=2}^n |f(\lambda_i)-f_{r_n}(\lambda_i)|.
    % \\ 
    % &\leq \frac{\delta}{2}+C_1(n-1)[(d_1-1)(d_2-1)]^{-\alpha'r_n}\leq \delta,
\end{align}
Suppose that all the non-trivial eigenvalues $\lambda\not=\lambda_1$ are contained in $[-3,3]$, from \eqref{eq:33estimate}, 
\begin{align*}
   \sum_{i=2}^n |f(\lambda_i)-f_{r_n}(\lambda_i)|&\leq C_1(n-1)[3(d_1-1)(d_2-1)^{-\alpha}]^{r_n+1}\\
   &\leq C_1n[(d_1-1)(d_2-1)^{-\alpha+2}]^{r_n}\leq C_1n^{1-\alpha'\beta}=o(1),
\end{align*}
which combining \eqref{eq:largen} implies for sufficiently large $n$,
\[ \left|Y_f^{(n)}-X_f^{(n)}\right|\leq \delta.\]

Recall \eqref{eq:gapspectral1} and the assumption $d_1\geq d_2$. With high probability, for a sequence $\epsilon_n\to 0$, we have the nontrivial eigenvalues of $\frac{XX^{\top}-d_1I}{\sqrt{(d_1-1)(d_2-1)}}$ is contained in 
\begin{align*}
    \left[ -2-\epsilon_n+\frac{d_2-2}{\sqrt{(d_1-1)(d_2-1)}},2+\epsilon_n+\frac{d_2-2}{\sqrt{(d_1-1)(d_2-1)}}\right]\subseteq [-3,3]
\end{align*} 
for sufficiently large $n$. Therefore
\begin{align*}
   \mathbb P\left( \left|Y_f^{(n)}-X_f^{(n)}\right|\geq \delta \right)\leq  \mathbb P\left( \max_{2\leq i\leq n}|\lambda_i|\geq 3 \right)=o(1).
\end{align*}
This finishes the proof.
\end{proof}

As a corollary of Theorem \ref{thm:CLTfixed}, we obtain   eigenvalue fluctuations  for the adjacency matrices of RBBGs as follows.
\begin{cor}
For fixed $d_1\geq d_2\geq 2$ and $(d_1,d_2)\not=(2,2)$, let $G_n$ be a sequence of  random $(d_1,d_2)$-biregular bipartite graph. Let $\lambda_1\geq \cdots \geq \lambda_{n+m}$ be the  eigenvalues of its adjacency matrix $A$. Suppose $f$ satisfies the same conditions as in Theorem \ref{thm:CLTfixed}. Then the random variable
\[Y_{f}^{(n)}:=\frac{1}{2}\left[\sum_{i=1}^{n+m} f\left(\frac{\lambda_i^2-d_1}{\sqrt{(d_1-1)(d_2-1)}}\right)-(m-n)f\left(\frac{-d_1}{\sqrt{(d_1-1)(d_2-1)}}\right) \right]-n a_0\]
converges in distribution as $n\to\infty$ to the infinitely divisible random variable
\[ Y_{f}:=\sum_{k=2}^{\infty} \frac{a_k}{[(d_1-1)(d_2-1)]^{k/2}}\CNBW_k^{(\infty)},\]
where $\CNBW_k^{(\infty)}$ is defined in \eqref{eq:CNBWinfinity}.
\end{cor}
\begin{proof}
Recall from Section \ref{sec:RBBG} that all eigenvalues of $A$ consist of two parts. There are $2n$ eigenvalues  in pair as $\{-\lambda,\lambda\}$ where $\lambda$ is a singular value of $X$. In addition, there are $(m-n)$ extra zero eigenvalues. the result then follows from the algebraic relation between eigenvalues of $A$ and eigenvalues of $XX^{\top}-d_1I$.
\end{proof}

\subsection{Gaussian fluctuations with growing degrees}\label{sec:gaussian}
In this section, we consider the eigenvalue fluctuations of RBBGs when $d_1\cdot d_2\to\infty$.

We first prove the following weak convergence result for a normalized and centered version of $\CNBW_k^{(\infty)
}$.

\begin{lemma}\label{lem:NKconvergenceZ}
Suppose that $d_1\cdot d_2\to\infty$, $r_n\to\infty$ as $n\to\infty$. For $k\geq 2$,
define 
\begin{align}\label{eq:Nkn1}
    N_k^{(n)}:=\frac{1}{[d_1-1)(d_2-1)]^{k/2}}\left(\CNBW_k^{(\infty)}-\mathbb E\CNBW_k^{(\infty)}\right) \mathbf{1}_{\{ k\leq r_n\}}.
\end{align}
Let $\{Z_k\}_{k\geq 2}$ be independent Gaussian random variables with $\mathbb E Z_k=0$ and $\mathbb E Z_k^2=2k$.   Define the weight $w_k=b_k/(k^2\log (k+1))$, where $(b_k)_{k\in\mathbb N}$ is any fixed positive summable sequence. 

Let $P_n$ be the law of the sequence $(N_k^{(n)})_{k\geq 2}$. Then 
as an element in $\mathcal P(\mathcal X)$, $ P_n$ converges weakly to the law of the random vector $(Z_k)_{k\geq 2}$. \end{lemma}

\begin{proof} We first prove the following 
\textit{Claim (1)}: for any fixed $r$, $(N_k^{(n)})_{2\leq k\leq r}$ converges in distribution to $(Z_k)_{2\leq k\leq r}$.

For any fixed $k$, when $n$ is sufficiently large,  we can write  \eqref{eq:Nkn1} as 
\begin{align} \label{eq:Nkn}
 N_k^{(n)}=&\frac{1}{[d_1-1)(d_2-1)]^{k/2}}\left(2kC_k^{(\infty)}-[(d_1-1)(d_2-1)]^{k}\right)\\
  &+\frac{1}{[d_1-1)(d_2-1)]^{k/2}}\sum_{j\mid k, j<k} \left(2jC_j^{(\infty)}-[(d_1-1)(d_2-1)]^{j}\right).\notag
\end{align}

Recall $C_k^{(\infty)}$ is a Poisson random variable with mean $\frac{(d_1-1)^k(d_2-1)^k}{2k}$.
The first term in \eqref{eq:Nkn} converges in distribution to a centered Gaussian random variable $Z_k$ with variance $2k$ as $n\to\infty$ from the Gaussian approximation of Poisson distribution. 

To show the convergence of $N_k^{(n)}$ for a fixed $k$, it remains to show the second term in \eqref{eq:Nkn} converges to zero in probability. Note that  the second term in \eqref{eq:Nkn} has a zero mean and its variance is given by
\begin{align*}
    \textnormal{Var}\left[ \frac{1}{[d_1-1)(d_2-1)]^{k/2}}\sum_{j\mid k, j<k} \left(2jC_j^{(\infty)}-[(d_1-1)(d_2-1)]^{j}\right)\right]=\sum_{j\mid k, j<k} 2j[(d_1-1)(d_2-1)]^{j-k},
\end{align*}
which goes to $0$ as $n\to\infty$. Then by Chebyshev's inequality, this term converges to $0$ in probability. Therefore Claim (1) holds.

We further define $N_1^{(n)}=0, Z_1=0$, and consider the weak convergence of $(N_k^{(n)})_{k\in \mathbb N}$ as an element in $L^2(\vec{w})$. Since
\begin{align*}
    \mathbb E\sum_{k=1}^{\infty}(Z_k)^2w_k=\sum_{k=2}^{\infty}\frac{b_k}{k\log (k+1)}<\infty,
\end{align*}
$(Z_k)_{k\in\mathbb N}\in L^2(\vec{w})$ almost surely. From Claim (1), every sub-sequential limit of $P_n$ has the same finite-dimensional distributions as $(Z_k)_{k\in \mathbb N}$. From Proposition \ref{prop:weakconvergence} (3),  every sub-sequential weak limit of $P_n$ in $\mathcal P(\mathcal X)$  is equal to the law of $(Z_k)_{k\in \mathbb N}$. 

By Prokhorov's Theorem (see for example \cite[Chapter 14, Theorem 1.5]{shorack2017probability}), if $\{P_n\}_{n\in\mathbb N}$ is tight, and every weakly convergent sub-sequence has the same limit $\mu$ in $\mathcal P(\mathcal X)$, then the sequence $\{P_n\}_{n\in\mathbb N}$ converges weakly to $\mu$. 
Since we have already shown  every sub-sequential weak limit of  $P_n$ is the law of $(Z_k)_{k\in \mathbb N}$ in $\mathcal P(\mathcal X)$, to finish the proof, it remains to show $\{P_n\}_{n\in\mathbb N}$ is tight. 

From the description of compact sets in $L^2(\vec{w})$ given in Proposition \ref{prop:weakconvergence} (1), it suffices to show for any $\varepsilon>0$, there exists an element $(a_k)_{k\in \mathbb N}\in L^2(\vec{w})$ with $a_k>0, \forall k\in \mathbb N$, such that
\begin{align}\label{eq:condddd}
    \sup_{n}\mathbb P\left[\bigcup_{k\in \mathbb N}\left\{|N_k^{(n)}|>a_k\right\}\right]=\sup_{n}\mathbb P\left[\bigcup_{k=1}^{r_n}\left\{|N_k^{(n)}|>a_k\right\}\right] <\varepsilon,
\end{align}
where $\bigcup_{k\in \mathbb N}\left\{|N_k^{(n)}|>a_k\right\}$ is the complement of a compact set in $L^2(\vec{w})$.

For any fixed $\epsilon>0$, choose $a_k=\alpha k \sqrt{\log (k+1)}$ for a constant $\alpha^2>32$ depending on $\varepsilon$, then $(a_k)_{k\in\mathbb N}\in L^2(\vec{w})$ and $a_k>0, \forall k\in \mathbb N$.
According to the definition of $N_k^{(n)}$ in \eqref{eq:Nkn1}, the above Condition \eqref{eq:condddd} is equivalent to 
\begin{align}
    \sup_{n}\mathbb P\left[\bigcup_{k=1}^{r_n}\left\{|\CNBW_{k}^{(\infty)}-\mathbb E\CNBW_{k}^{(\infty)}|>a_k[(d_1-1)(d_2-1)]^{k/2}\right\}\right] <\varepsilon.
\end{align}

From the proof of Theorem 22 in \cite{dumitriu2013functional}, $\CNBW_{k}^{(\infty)}$, as a sum of independent Poisson random variables, satisfies the following concentration inequality: for any $t>0$,
\begin{align}\label{eq:poissonconcentration}
   \mathbb P\left( |\CNBW_{k}^{(\infty)}-\mathbb E\CNBW_{k}^{(\infty)}|>t\right)\leq 2\exp\left(-\frac{t}{8k}\log\left(1+\frac{t}{2k[(d_1-1)(d_2-1)]^k} \right) \right).
\end{align}

Since $\log (1+x)\geq x/2$ for $x\in [0,1]$, we have from \eqref{eq:poissonconcentration}, for sufficiently large $n$ and  all $k\leq r_n$,
\begin{align*}
    &\mathbb P\left( |\CNBW_{k}^{(\infty)}-\mathbb E\CNBW_{k}^{(\infty)}|>a_k[(d_1-1)(d_2-1)]^{k/2}\right)\\
   \leq  &2\exp\left(-\frac{a_k[(d_1-1)(d_2-1)]^{k/2}}{8k}\log\left(1+\frac{a_k}{2k[(d_1-1)(d_2-1)]^{k/2}} \right) \right) \\
   \leq &2 \exp\left(-\frac{a_k^2}{32k^2}\right)=2(k+1)^{-\alpha^2/32}.
\end{align*}
With the assumption $\alpha^2>32$, we can  make
$\sum_{k=1}^{\infty}2(k+1)^{-\alpha^2/32}<\epsilon$ by choosing a sufficiently large constant $\alpha$ depending on $\varepsilon$, which guarantees \eqref{eq:condddd}. Hence $\{P_n\}_{n\in \mathbb N}$ is tight. This completes the proof. 
\end{proof}

We now continue to study the eigenvalue fluctuation for $\frac{XX^{\top}-d_1I}{\sqrt{(d_1-1)(d_2-1)}}$ when $d_1\cdot d_2\to\infty$. Before stating the main result, we make several assumptions on the test function $f$.
Define 
\begin{align*}
    \Phi_0(x) &=1,\quad 
    \Phi_k(x)=2T_k(x/2), \quad \forall k\geq 1.
\end{align*}
Assume $f$ is an entire function on $\mathbb C$.  Let $K_1=\max\{\alpha_1,\alpha_2 \}$, where $\alpha_1$ and $\alpha_2$ are the constants in \eqref{eq:E1}, \eqref{eq:E2}, respectively.
Then from Proposition \ref{prop:approx}, $f$ has the expansion
\begin{align}\label{eq:fphi}
   f(x)=\sum_{i=0}^{\infty} a_i \Phi_i(x) 
\end{align}
on $[-K_1,K_1]$. Denote 
\[
    f_k(x):=\sum_{i=0}^k a_i\Phi_i(x). 
\]

Suppose the following conditions hold for $f$:
\begin{enumerate}
    \item For some $\alpha>3/2$ and $M>0$, 
    \begin{align}\label{eq:expbound}
        \sup_{|x|\leq K_1}|f(x)-f_{k}(x)|\leq M\exp(-\alpha k h(k)),
    \end{align}
    where $h$ is a function such that $h(r_n)\geq \log [(d_1-1)(d_2-1)]$ for a sequence 
     \begin{align} \label{eq:rnbound}
    r_n &=\left\lfloor\frac{\beta \log n}{\log[(d_1-1)(d_2-1)]} \right\rfloor
\end{align}
with a constant $\beta<1/\alpha$.
\item   
\begin{align}\label{eq:frn}
    \lim_{n\to\infty} \left| f_{r_n}\left(\frac{d_1(d_2-1)}{\sqrt{(d_1-1)(d_2-1)}}\right)-f\left(\frac{d_1(d_2-1)}{\sqrt{(d_1-1)(d_2-1)}}\right)\right|=0.
\end{align}
\end{enumerate}

Let 
$\mu_k(d_1,d_2):=\mathbb E\CNBW_k^{(\infty)}.$
We define the following sequence: 
\begin{align}\label{eq:defmfn}
    m_f^{(n)}:=na_0+\sum_{k=1}^{r_n} \frac{a_k}{[(d_1-1)(d_2-1)]^{k/2}}\left(\mu_k(d_1,d_2)-n(d_1-2)\cdot (d_2-1)^{k/2}  \mathbf{1}\{ \text{$k$ is even}\}\right).
\end{align}

Now we are ready to state our results for  eigenvalue fluctuations when $d_1\cdot d_2\to\infty$. Here $d_1,d_2$, $(\lambda_i)_{1\leq i\leq n}$ and the matrix $X$ are quantities  depending on $n$, but for simplicity of notations, we drop the dependence on $n$. 

\begin{theorem}\label{thm:GaussianCLT}
Let $G_n$ be a sequence of random $(d_1,d_2)$-biregular bipartite graphs with 
\[d_1d_2\to\infty,\quad  d_1d_2=n^{o(1)}.\] Let $\lambda_1\geq \cdots\geq \lambda_n$ be the eigenvalues of $\frac{XX^{\top}-d_1I}{\sqrt{(d_1-1)(d_2-1)}}$. 
Suppose one of the following two assumptions holds:
\begin{enumerate}
    \item  There exists a constant $c\geq 1$ such that $1\leq \frac{d_1}{d_2}\leq c$.
    \item There exists a constant $c_1$ such that $d_2\leq c_1$ for all $n$.
\end{enumerate} 
Let $f$ be an entire function on $\mathbb C$ satisfying \eqref{eq:expbound} and \eqref{eq:frn}.
Then as $n\to\infty$, the random variable
\begin{align}\label{eq:Yfn}
  Y_f^{(n)}=\sum_{i=1}^n f(\lambda_i)- m_f^{(n)}  
\end{align}
converges in distribution to a centered Gaussian random variable with  variance $\sigma_f=2\sum_{k=2}^{\infty}ka_k^2$.

Moreover, for any fixed $t$, consider the entire functions $g_1,\dots, g_t$   satisfying \eqref{eq:expbound} and \eqref{eq:frn}. The corresponding random vector 
$(Y_{g_1}^{(n)},\dots, Y_{g_t}^{(n)})$ converges in distribution to a centered Gaussian random vector $(Z_{g_1},\dots, Z_{g_t})$ with covariance
\begin{align}\label{eq:cov}
    \textnormal{Cov}(Z_{g_i},Z_{g_j})=2\sum_{k=2}^{\infty}ka_k(g_i) a_k(g_j)
\end{align}
for $1\leq i,j\leq t$,
where $a_k(g_i), a_k(g_j)$ are the $k$-th coefficients in the expansion \eqref{eq:fphi} for $g_i,g_j$, respectively.
\end{theorem}

\begin{proof}
We first prove the CLT for a single test function $f$.
Define 
\begin{align*}
    X_f^{(n)} &:=\sum_{k=2}^{r_n}\frac{a_k}{[(d_1-1)(d_2-1)]^{k/2}}\CNBW_{k}^{(n)} -\mathbb E\sum_{k=2}^{r_n}\frac{a_k}{[(d_1-1)(d_2-1)]^{k/2}} \CNBW_{k}^{(\infty)} ,\\
     \tilde{X}_f^{(n)} &:=\sum_{k=2}^{r_n}\frac{a_k}{[(d_1-1)(d_2-1)]^{k/2}}\CNBW_{k}^{(\infty)} -\mathbb E\sum_{k=2}^{r_n}\frac{a_k}{[(d_1-1)(d_2-1)]^{k/2}} \CNBW_{k}^{(\infty)}.
\end{align*}

Recall the definition of $m_f(n)$ in \eqref{eq:defmfn}. 
From  \eqref{eq:gamma9}, \eqref{eq:gamma99}, and  \eqref{eq:formulaCNBW}, $X_f^{(n)}$ can be written as
\begin{align*}
    X_f^{(n)}&=\sum_{k=2}^{r_n}\sum_{i=1}^n a_k\Gamma_k(\lambda_i)-\sum_{k=2}^{r_n}\frac{a_k}{[(d_1-1)(d_2-1)]^{k/2}}\mu_k(d_1,d_2)\\
    &=\sum_{k=2}^{r_n}\sum_{i=1}^n \left(2a_k T_k(\lambda_i/2)+\frac{a_k(d_1-2)}{(d_1-1)^{k/2}}\mathbf{1}_{\{\text{$k$ is even}\}}\right)-\sum_{k=2}^{r_n}\frac{a_k}{[(d_1-1)(d_2-1)]^{k/2}}\mu_k(d_1,d_2)\\
    &=\sum_{i=1}^n f_{r_n}(\lambda_i)-na_0+ \sum_{k=2}^{r_n}\frac{na_k(d_1-2)}{(d_1-1)^{k/2}}\mathbf{1}_{\{\text{$k$ is even}\}}-\sum_{k=2}^{r_n}\frac{a_k}{[(d_1-1)(d_2-1)]^{k/2}}\mu_k(d_1,d_2)\\
    &=\sum_{i=1}^n f_{r_n}(\lambda_i)-m_f^{(n)},
\end{align*}
where in the third line we use the fact given in \eqref{eq:formulaCNBW} that 
\[\sum_{i=1}^n 2a_1T_1(\lambda_i/2)=\sum_{i=1}^n a_1\Gamma_1(\lambda_i)=a_1[(d_1-1)(d_2-1)]^{-1/2}\CNBW_1^{(n)}=0.\]
From the definition of $N_k^{(n)}$ in \eqref{eq:Nkn1},
\begin{align*}
    \tilde{X}_f^{(n)}&=\sum_{k=2}^{r_n}a_kN_k^{(n)}.
\end{align*}
By Lemma \ref{lem:NKconvergenceZ} and Proposition \ref{prop:weakconvergence} (2), $\tilde{X}_f^{(n)}$ converges in distribution to a centered Gaussian random variable with variance  $\sigma_f=\sum_{k=2}^{\infty}2ka_k^2$. 

From Corollary \ref{cor:dTVCNBW}, the total variation distance between $X_f^{(n)}$ and $\tilde{X}_f^{(n)}$ satisfies 
\begin{align*}
    d_{\textnormal{TV}}(X_f^{(n)}, \tilde{X}_f^{(n)})&\leq 	d_{\textnormal{TV}}\left( (\CNBW_k^{(n)},2\leq k\leq r_n),(\CNBW_{k}^{(\infty)},2\leq k\leq r_n)\right)\\
    &\leq \frac{c_8\sqrt{r_n}[(d_1-1)(d_2-1)]^{3r_n/2}}{nd_1},
\end{align*}
which converges to $0$ as $n\to\infty$ from the assumption \eqref{eq:rnbound}. Therefore $X_f^{(n)}$ and  $\tilde{X}_f^{(n)}$ converge to the same limit.

It remains to show  $Y_f^{(n)}$ and $X_f^{(n)}$ converge in distribution to the same limit. We have $f_{r_n}(\lambda_1)\to f(\lambda_1)$ as $n\to\infty$ from \eqref{eq:frn}. Then for any $\delta>0$, 
$|f(\lambda_1)-f_{r_n}(\lambda_1)|\leq \delta/2$ for sufficiently large $n$.

Suppose that all the non-trivial eigenvalues are contained in $[-K_1,K_1]$. From Condition \eqref{eq:expbound},  we have for sufficiently large $n$,
\begin{align*}
    \left|Y_f^{(n)}-X_f^{(n)}\right|&\leq \sum_{i=1}^n |f(\lambda_i)-f_{r_n}(\lambda_i)|\leq \frac{\delta}{2}+ (n-1)M\exp(-\alpha r_nh(r_n))\leq \frac{\delta}{2}+Mn^{1-\alpha\beta}<\delta.
\end{align*}
Therefore 
\begin{align}\label{eq:Yfff}
  \mathbb P\left( \left|Y_f^{(n)}-X_f^{(n)}\right|\geq \delta \right)\leq  \mathbb P\left( \max_{2\leq i\leq n}|\lambda_i|\geq K_1 \right)=o(1),
\end{align}
where the last inequality is from parts (2) and (3) in Theorem \ref{thm:spectralgap}. Hence $Y_f^{(n)}$ and $X_f^{(n)}$ converge in distribution to the same limit.  This proves the CLT for \eqref{eq:Yfn}.

We now extend the results to a  random vector $(Y_{g_1},\dots, Y_{g_t})$. By Lemma \ref{lem:NKconvergenceZ} and part (2) in Proposition \ref{prop:weakconvergence}, the random vector $(\tilde{X}_{g_1}^{(n)},\dots, \tilde{X}_{g_t}^{(n)})$ converges in distribution to the Gaussian random  vector $(Z_{g_1},\dots, Z_{g_t})$ with covariance given in  \eqref{eq:cov}.
 
Note that each entry in the vector $(X_{g_1}^{(n)},\dots, X_{g_t}^{(n)} )$ is a measurable function of $ (\CNBW_k^{(n)} )_{2\leq k\leq r_n}$, and we can find a measurable map $\psi$   such that 
\begin{align*}
 \psi( (\CNBW_k^{(n)} )_{2\leq k\leq r_n})&=(X_{g_1}^{(n)},\dots, X_{g_t}^{(n)}), \quad 
 \psi( (\CNBW_k^{(\infty)} )_{2\leq k\leq r_n})=(\tilde X_{g_1}^{(n)},\dots, \tilde{X}_{g_t}^{(n)}).
\end{align*}
Since any measurable map reduces the total variation distance between two random variables, we obtain
from \eqref{eq:TVinequality},
\begin{align*}
    &d_{\textnormal{TV}}\left( (X_{g_1}^{(n)},\dots, X_{g_t}^{(n)}),(\tilde{X}_{g_1}^{(n)},\dots, \tilde{X}_{g_t}^{(n)})\right)\\
    \leq & d_{\textnormal{TV}}\left( (\CNBW_k^{(n)},2\leq k\leq r_n),(\CNBW_{k}^{(\infty)},2\leq k\leq r_n)\right)\\
    \leq & \frac{c_8\sqrt{r_n}[(d_1-1)(d_2-1)]^{3r_n/2}}{nd_1}=o(1). 
\end{align*}
Therefore $\left(X_{g_1}^{(n)},\dots, X_{g_t}^{(n)}\right)$ converges in distribution to $(Z_{g_1},\dots, Z_{g_t})$. 
Finally, according to \eqref{eq:Yfff}, $\left(X_{g_1}^{(n)},\dots, X_{g_t}^{(n)}\right)$ and $\left(Y_{g_1}^{(n)},\dots, Y_{g_t}^{(n)}\right)$ converge in distribution to the same limit. This finishes the proof.
\end{proof} 

\begin{remark}
 In \cite{chen2015clt}, the authors proved a CLT for linear spectral statistics for normalized sample covariance matrices $A=\frac{1}{\sqrt{np}}(XX^{\top}-pI)$, where $p/n\to\infty$ and $X=(X_{ij})_{n\times p}$ has i.i.d. entries with mean $0$ variance $1$. It is shown in Theorem 1 of \cite{chen2015clt} that the fluctuations of linear statistics for two analytic functions $g_1,g_2$ converge in distribution to a centered Gaussian vector with covariance given by $(\nu_4-3)a_1(g_1)a_1(g_2)+2\sum_{k=1}^{\infty}ka_k(g_1)a_k(g_2),$
 where  $\nu_4=\mathbb EX_{11}^4$.  
The covariance  given in \eqref{eq:cov} is  the same, except for the fact that the  coefficient in front of $a_1(f_1)a_1(f_2)$ is 0. This can be explained by the fact that the number of 2-cycles  is 0 in RBBGs, whereas in the model used in \cite{chen2015clt} it is not.
 The same phenomenon was also observed in uniform  random regular graphs \cite{johnson2015exchangeable}, where the limiting variance  is the same as the eigenvalue fluctuations for the GOE  except for the first two terms, see Remark 22 in \cite{johnson2015exchangeable}.
\end{remark}

\section{Global semicircle law}\label{sec:globallaw}

Consider a random $(n,m,d_1,d_2)$-biregular bipartite graph with $d_1\geq d_2$. We assume $d_1,d_2$ satisfy the following: 
\begin{align}
	\lim_{n\to\infty}d_1 &=\infty,\label{eq:assumption1}\\
	d_1 &=o(n^{\epsilon}), \quad \forall \epsilon>0, \label{eq:assumption2}\\
	\frac{d_1}{d_2} &\to\infty.\label{eq:assumption3} 
\end{align}
Here $d_2$ can be fixed or a parameter depending on $n$. 
 In this section, we prove a semicircle law for the matrix $\frac{XX^{\top}-d_1I}{\sqrt{(d_1-1)(d_2-1)}}$ under the assumptions \eqref{eq:assumption1}-\eqref{eq:assumption3}.
 
 For RBBGs in this regime, we have the locally tree-like structure in the following sense. 
Let $R$ be fixed and $\tau_1$ be the set of vertices in $V_1$ without any cycles in the $R$-neighborhood. The following lemma holds.
\begin{lemma}\label{taubound}
	Then under Condition
	\eqref{eq:assumption2},
	\begin{align*}
% 	\mathbb P\left( \frac{m-|\tau_2|}{n+m}>n^{-1/4}\right)=o(n^{-5/4}),\\
		\mathbb P\left( \frac{n-|\tau_1|}{n}>n^{-1/4}\right)=o(n^{-5/4}).
	\end{align*}
\end{lemma}
To prove Lemma \ref{taubound}, the following estimates on the expectation and variance of the cycle counts  of RBBGs given in \cite{dumitriu2016marvcenko} are needed.
\begin{lemma}[Proposition 4 in \cite{dumitriu2016marvcenko}]\label{marvcenko} Let $C_k$ be the number of cycles of length $2k$ in a random  $(d_1,d_2)$-biregular bipartite graph. Denote $\mu_k=\frac{[(d_1-1)(d_2-1)]^{k}}{2k}$. If $d_1=o(n), k=O(\log n)$ and $kd_1=o(n)$, then 
    \begin{align}
        \mathbb EC_k&=\mu_k\left(1+O\left( \frac{k(k+d_1)}{n}\right)\right),\\
        \textnormal{Var}[C_k]&=\mu_k\left(1+O\left( \frac{d_1^{2k}(k(d_1/d_2)^{2k-1}+(d_1/d_2)^{-k}d_2)}{n}\right)\right).
    \end{align}
\end{lemma}

\begin{proof}[Proof of Lemma \ref{taubound}] 

From Lemma \ref{marvcenko}, for each fixed $k$, under Condition \eqref{eq:assumption2},
\begin{align}\label{eq:meanvariance}
    \mathbb EC_k=(1+o(1))\mu_k,\quad \textnormal{Var}[C_k]=(1+o(1))\mu_k.
\end{align} 

If a vertex $v_1\in V_1$ is not in $\tau_1$, then for some $s$ with $2\leq s\leq R$, there exists a $2s$-cycle within $(R-s)$-neighborhood of $v_1$. Hence the size of all $(R-s)$-neighborhoods of $2s$-cycles from $V_1$ gives an upper bound on $(n-|\tau_1|)$. 

For any $2s$-cycle, the size of its $(R-s)$-neighborhood from $V_1$ is bounded by \[c_1s[(d_1-1)(d_2-1)]^{(R-s)/2+1}\] with an absolute constant $c_1$. Define
\begin{align*}
    N_R:=c_1\sum_{s=2}^Rs[(d_1-1)(d_2-1)]^{(R-s)/2+1}C_s.
\end{align*}
We then have $n-\tau_1\leq N_R.$ From \eqref{eq:meanvariance},
$ \mathbb EN_R=O([(d_1-1)(d_2-1)]^{R+1}). $

Recall $R$ is fixed. By Cauchy inequality,
\begin{align*}
    \textnormal{Var}[N_R]\leq c_1^2R\sum_{s=2}^Rs^2[(d_1-1)(d_2-1)]^{R-s+2}\textnormal{Var}[C_s]=O([(d_1-1)(d_2-1)]^{R+2}).
\end{align*}
Then from Markov's inequality, together with our assumptions \eqref{eq:assumption1}-\eqref{eq:assumption3},
\begin{align*}
    \mathbb P\left(\frac{n-|\tau_1|}{n}>n^{-1/4}\right)&=\mathbb P( n-|\tau_1| >n^{3/4}) \leq \mathbb P(N_R\geq n^{3/4})\\
    &\leq \frac{\mathbb E[N_R^2]}{n^{3/2}}=O([(d_1-1)(d_2-1)]^{2R+2}n^{-3/2})=o(n^{-5/4}).
\end{align*}
\end{proof}

We now state our main result in this section. The proof is based on the moment method and the tree approximation of local neighborhoods, which were previously applied to random regular graphs  in \cite{dumitriu2012sparse}. 
\begin{theorem}\label{thm:globallaw}
Let $G_n$ be a sequence of  random $(d_1,d_2)$-biregular bipartite graph. Under assumptions \eqref{eq:assumption1}-\eqref{eq:assumption3},  the empirical spectral distribution of $\frac{XX^{\top}-d_1I}{\sqrt{(d_1-1)(d_2-1)}}$ converges weakly to the semicircle law almost surely.
\end{theorem}
\begin{remark}
Recall in \cite{dumitriu2016marvcenko}, when the ratio $d_1/d_2\geq 1$ converges to a positive constant, the ESD of $\frac{XX^{\top}}{d_1}$ converges to Mar\v{c}enko-Pastur law. With different scaling parameters, we obtain a different semicircle law when $d_1/d_2\to\infty$. This can be seen as an analog of the semicircle law for sample covariance matrices proved in \cite{bai1988convergence} when the aspect ratio is unbounded.
\end{remark}

\begin{proof}[Proof of Theorem \ref{thm:globallaw}]
Note that for all $i\in V_1$, by the degree constraint,
\begin{align}\label{eq:diagonal}
(XX^{\top})_{ii}=\sum_{j}X_{ij}X_{ji}=\sum_{j}X_{ij}=\deg (i)=d_1.
\end{align}
	Denote $M=\frac{XX^{\top}-d_1I}{\sqrt{(d_1-1) (d_2-1)}}$. We start with the trace expansion of $M$. 
	\begin{align}
	\frac{1}{n}\textnormal{tr} M^{k}&=\frac{1}{n((d_1-1)(d_2-1))^{k/2}}\textnormal{tr}(XX^{\top}-d_1I)^k \notag\\
	&=	\frac{1}{n((d_1-1)(d_2-1))^{k/2}}\sum_{\substack{ i_1,\dots, i_k\in [n]\\ i_1\not=i_2,\dots, i_{k}\not=i_1\\ j_1,\dots, j_k\in [m]}}X_{i_1j_1}X_{i_2j_1}\cdots X_{i_{k}j_{k}}X_{i_1j_{k}}.\label{eq:traceexpansion}
	\end{align}
	From \eqref{eq:diagonal}, the diagonal entries of $XX^T-d_1I$ are 0, therefore we have the constraint that $i_1\not=i_2,\dots, i_k\not=i_1$ in \eqref{eq:traceexpansion}.
	
	Let $A_k^{r,c}(v,v)$ be the number  of all
	closed  walks of  length $2k$ in $G$ starting from $v\in V_1$ that use $r$ distinct vertices from $V_1$, $c$ distinct vertices from $V_2$, with the restriction that $i_1\not=i_2,\dots, i_{k}\not=i_1$. We have  $r\leq k+1$ and $c\leq k$, since there are at most  $k+1$ vertices in $V_1$ and $k$ vertices in $V_2$ that are visited in one closed walk of length $2k$.  From \eqref{eq:traceexpansion}, the $k$-th moment of the empirical spectral distribution $\mu_n$  satisfies
\begin{align}
	\int x^k d\mu_n(x)=\frac{1}{n}\textnormal{tr} M^{k}=\frac{1}{n((d_1-1)(d_2-1))^{k/2}}\sum_{v\in V_1}\sum_{r=1}^{k+1}\sum_{c=1}^k A_k^{r,c}(v,v).\label{eq:momentcount}
\end{align}
Since $d_1\geq d_2$, for any fixed $v\in V_1$, we have  
\[\sum_{r\leq k+1,c\leq k} A_k^{r,c}(v,v)\leq d_1^{2k}.\]
For ease of notation,  in the following equations, we often omit the range of $r,c$ in the summation.

We may decompose the sum in \eqref{eq:momentcount} into two parts depending on whether $v\in \tau_1$ or not.
For any  $v\in\tau_1$, we write
$A_k^{r,c}=:A_k^{r,c}(v,v)$ since all neighborhoods of $v\in \tau_1$ of radius $k$ looks the same and the number of such closed walks is independent of $v$. 
Now we have the following upper bound on \eqref{eq:momentcount}:  
\begin{align*}
	\int x^k d\mu_n(x) &\leq \frac{1}{n((d_1-1)(d_2-1))^{k/2}}\sum_{v\in\tau_1}\sum_{r,c}A_k^{r,c}(v,v)+ \frac{(n-|\tau_1|) d_1^{2k}}{n((d_1-1)(d_2-1))^{k/2}}\\
	&=\frac{|\tau_1|}{n((d_1-1)(d_2-1))^{k/2}} \sum_{r,c}A_k^{r,c}+\frac{(n-|\tau_1|) d_1^{2k}}{n((d_1-1)(d_2-1))^{k/2}}\\
	&\leq \frac{1}{((d_1-1)(d_2-1))^{k/2}} \sum_{r,c}A_k^{r,c}+\frac{(n-|\tau_1|) d_1^{2k}}{n((d_1-1)(d_2-1))^{k/2}}.
\end{align*}
Similarly, a lower bound holds by only counting closed walks starting with vertices in $\tau_1$:
\begin{align*}
	\int x^k d\mu_n(x)\geq & \frac{1}{n((d_1-1)(d_2-1))^{k/2}}\sum_{v\in \tau_1}\sum_{r}\sum_{c} A_k^{r,c}(v,v) 
	= \frac{|\tau_1|}{n((d_1-1)(d_2-1))^{k/2}}\sum_{r}\sum_{c} A_k^{r,c}.
\end{align*}

From Lemma \ref{taubound} and assumption \eqref{eq:assumption2}, with probability at least $1-o(n^{-5/4})$, for any fixed $k\geq 0$,
\[ \frac{(n-|\tau_1|)}{n((d_1-1)(d_2-1))^{k/2}} d_1^{2k}=o(1), \quad\text{and} \quad  \frac{|\tau_1|}{n}=1-o(n^{-1/4}).\] To show the almost sure convergence of the empirical measure to semicircle law, by the upper and lower bounds above, it suffices to show 
\begin{align}\label{catalan}
	\lim_{n\to\infty}\frac{1}{((d_1-1)(d_2-1))^{k/2}}\sum_{r,c} A_k^{r,c} =\begin{cases}
	0 & \text{ if $k$ is odd,}\\
	C_{k/2}  &\text{ if $k$ is even,}
	\end{cases}
\end{align}
where $C_k:=\frac{1}{k+1}{\binom{2k}{k}}$ is the $k$-th Catalan number.

Recall $A_{k}^{r,c}$ counts the closed walks of length $2k$ on a rooted $(d_1,d_2)$-biregular tree starting from a root with degree $d_1$, ending at the same  root. Now we consider the quantity 
%It remains to estimate \eqref{catalan} for even $k$. 
\[\displaystyle
	 \frac{1}{((d_1-1)(d_2-1))^{k/2}}\sum_{r,c} A_{k}^{r,c}\]
  more carefully. We first consider possible ranges of $r$ and $c$ in the expression above. 
  
  The walk $(i_1,j_1,i_2,j_2,\dots,i_k,j_k,i_1)$  in the summation satisfies $i_1\not=i_2,\cdots,i_{k-1}\not=i_k, i_k\not=i_1.$
This implies when a walk goes from $i_t$ to $j_t$ for some $t$, it cannot backtrack immediately to $i_t$. Namely, any such walk is not allowed to backtrack at even depths (here, we define the depth of the root in a tree as 1). To have a closed walk of length $2k$ on a tree, each edge is repeated at least twice, so the number of distinct edges is at most $k$. Therefore the number of distinct vertices satisfies 
\begin{align}\label{eq:r+c}
    r+c\leq k+1.
\end{align}  
 
 For fixed $r$ and $c$, the number of such unlabeled rooted trees with $r+c-1$ distinct edges is  $C_{r+c-1}$.  Let $I$ be the set of vertices in the odd depths of the biregular tree and $J$ be the set of vertices in the even depths. Since the first vertex of the walk is fixed (we always start from the fixed root), for any closed walk, there are at most $d_1^c$ many ways to choose distinct vertices from $J$ and $d_2^{r-1}$ many ways to choose distinct vertices from $I$.
 Therefore we have
\begin{align}\label{eq:arc}
A_k^{r,c}\leq d_1^{c} d_2^{r-1}C_{r+c-1}\leq d_1^cd_2^{r-1}C_k,
\end{align}
where the last inequality is from \eqref{eq:r+c}.
We also know that $r-1\geq c$, because whenever  a new vertex in $J$ is reached by the walk, the walk cannot backtrack, so it must reach a new vertex in $I$. Therefore we have 
$$c\leq r-1 \quad \text{and}\quad   r+c\leq k+1,$$ which implies the following conditions on $c$ and $r$:
\begin{align}\label{eq:range}
c\leq k/2 \quad \text{and}\quad  r-1\leq k-c.	
\end{align}

From \eqref{eq:arc}, for any $(r,c)$ satisfying \eqref{eq:range}, the following holds:
\begin{align}
	\frac{A_k^{r,c}}{((d_1-1)(d_2-1))^{k/2}}\leq \frac{d_1^{c}}{(d_1-1)^{k/2}}\frac{d_2^{r-1}}{(d_2-1)^{k/2}}C_k\leq \frac{d_1^{c}}{(d_1-1)^{k/2}}\frac{d_2^{k-c}}{(d_2-1)^{k/2}}C_k.\end{align}

 Now we discuss two cases depending on the parity of $k$.
 When $k$ is odd, from \eqref{eq:range}, $c\leq \frac{k-1}{2}$.  
 Since $d_1/d_2\to\infty$, we obtain 
 \begin{align} \label{eq:odd} 
	  \frac{1}{((d_1-1)(d_2-1))^{k/2}}\sum_{r,c} A_{k}^{r,c}\leq & \left(\frac{d_1}{d_2}\right)^c\frac{d_2^kC_k}{[(d_1-1)(d_2-1)]^{k/2}}
	  \leq  \left(\frac{d_1}{d_2}\right)^{c-k/2} 2^k C_k=o(1).
\end{align}

 When $k$ is even, to have a non-vanishing term in the limit for $A_k^{r,c}$, we must have $c=k/2$ and $r=k/2+1$. Then we have 
 \begin{align}\label{eq:leadingterm}
	 \frac{1}{((d_1-1)(d_2-1))^{k/2}}\sum_{r,c} A_{k}^{r,c}=\frac{1}{((d_1-1)(d_2-1))^{k/2}}	A_k^{k/2+1, k/2}+o(1).
\end{align} 

We  continue our proof with a  more refined estimate on $A_k^{k/2+1,k/2}$. Since every  edge is repeated exactly twice in the closed walk, it's a depth-first search on the biregular tree. 

If the root is at level $1$, and subsequent vertices are at a level $i+1$ where $i$ is the distance from the root, then all leaves must be at odd levels, since we can never backtrack at an even level. This implies that every vertex at an even level has at least one child, which means $r \geq c+1$, with equality if and only if every vertex at an even level has \emph{exactly} one child. Thus, one can see the tree as a subdivision of a smaller tree, where a vertex has been introduced on each edge (the ``new" vertices being the vertices on an even level in the bigger tree). This is a bijection between the kind of planar rooted tree on $k+1$ vertices we are trying to count and the set of all planar rooted trees on $k/2+1$ vertices. There are $C_{k/2}$ of the latter. See Figure  \ref{fig:depthfirst} for an example of a valid closed walk and an illustration of the aforementioned bijection.

\begin{figure}[ht]
    \centering
    \includegraphics{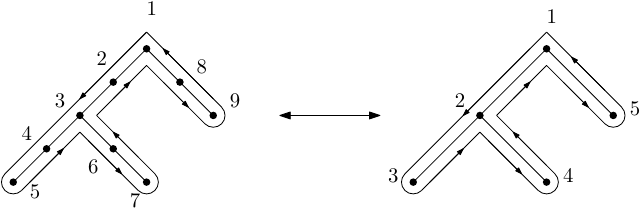}
    \caption{On the left we have a closed walk $(1,2,3,4,5,4,3,6,7,6,3,2,1,8,9,8,1)$ on a rooted planar tree which only backtracks at odd depths, and the tree has no new branches at any even depth along the walk. Its correspondent under the bijection is the closed walk  $(1,2,3,2,4,2,1,5,1)$ on the smaller rooted planar tree induced by the depth-first search on the right.}
    \label{fig:depthfirst}
\end{figure}

Moreover, given a fixed root with a  vertex label, the number of all possible ways to label the tree with vertices in a biregular bipartite graph is between $d_1^{k/2}(d_2-1)^{k/2}$ and $(d_1-k/2)^{k/2} (d_2-1)^{k/2}$, so  the following inequality for $A_{k}^{k/2+1, k/2}$ holds:
\begin{align}\label{eq:AKKKKK}
	(d_1-k/2-1)^{k/2} (d_2-1)^{k/2}C_{k/2}\leq A_{k}^{k/2+1, k/2}\leq d_1^{k/2}(d_2-1)^{k/2}C_{k/2}.
\end{align}
From \eqref{eq:leadingterm} and \eqref{eq:AKKKKK}, we obtain for even $k$,
\begin{align}\label{eq:even}
	\lim_{n\to\infty}\frac{1}{((d_1-1)(d_2-1))^{k/2}}\sum_{r,c} A_{k}^{r,c}=C_{k/2}.
\end{align}
 With \eqref{eq:odd} and \eqref{eq:even},  the asymptotic behavior of moments given in  \eqref{catalan} holds. This completes the proof of Theorem \ref{thm:globallaw}.
\end{proof}

% \subsection{Eigenvalue fluctuations for the adjacency matrix}

% By the algebraic relation between eigenvalues of $A$ and $XX^{\top}$, we get the following immediate corollary.
% \begin{cor}
% Let $G$ be a random $(d_1,d_2)$-biregular bipartite graph with $d_1d_2\to\infty$ as $n\to\infty$. Let $\lambda_1\geq \cdots\geq \lambda_n$ be the eigenvalues of $A$. Assume $f$ is entire with order less than $s$, which implies that $f$ can be written as $f(x)=\sum_{k=0}^{\infty}a_k T_k(x/2)$ and  $d_1d_2\leq (\log n)^{\frac{2}{3s}-\epsilon}$ for some $\epsilon>0$.
% If one of the two assumptions holds,
% \begin{enumerate}
%     \item  there exists a constant $c\geq 1$ such that $1\leq \frac{d_1}{d_2}\leq c$.
%     \item $d_2\leq c_1$ for a constant $c_1$.
% \end{enumerate} then
% \[\sum_{i=1}^{n+m} f\left(\frac{\lambda_i^2-d_1}{\sqrt{(d_1-1)(d_2-1)}}\right)-\sum_{i=1}^{n+m}\mathbb E f\left(\frac{\lambda_i^2-d_1}{\sqrt{(d_1-1)(d_2-1)}}\right)\]
% converges in distribution to a normal random variable with mean zero and variance $2\sum_{k=3}^{\infty}ka_k^2$.
% \end{cor}

\section{Random regular hypergraphs}\label{sec:hypergraph}

We first include some definitions for  hypergraphs and describe a bijection between a subset of biregular bipartite graphs and the set of regular hypergraphs studied in \cite{dumitriu2019spectra}.
We will use the map given in Definition \ref{def:incidence} to apply some of our results for RBBGs to random regular hypergraphs, see \cite{dumitriu2019spectra} for more details. 

\begin{definition}[hypergraph]
	A \textit{hypergraph} $H$ consists of a set $V$ of vertices and a set $E$ of hyperedges such that each hyperedge is a nonempty set of $V$.  A hypergraph $H$ is \textit{$k$-uniform} for an integer $k\geq 2$ if every hyperedge $e\in E$ contains exactly $k$ vertices. The \textit{degree} of $i$, denoted   $\deg(i)$, is the number of all hyperedges incident to $i$.
 A hypergraph is \textit{$d$-regular} if all of its vertices have degree $d$.
A hypergraph is \textit{$(d,k)$-regular} if it is both $d$-regular and $k$-uniform. 
\end{definition}

\begin{definition}[incidence matrix and associated bipartite graph]\label{def:incidence}
 A vertex $i$ is \textit{incident} to a hyperedge $e$ if and only $v$ is an element of $e$. We can define the \textit{incidence matrix} $X$ of a hypergraph $H=(V,E)$ to be a $|V|\times |E|$ matrix indexed by elements in $V$ and $E$ such that $X_{i,e}=1$ if $i\in e$ and $0$ otherwise. Moreover, if we regard $X$ as the adjacency matrix of a graph, it defines a bipartite graph $G$ with two vertex sets  $V$ and $E$. We call $G$  the  \textit{bipartite graph associated to $H$}, given by a map $\Phi$ (so $\Phi(H)=G$).
See Figure \ref{fig:bijection} for an example.
 \end{definition}
 
 \begin{figure}[ht]
\includegraphics{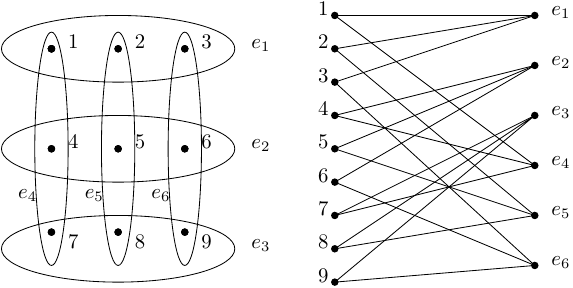}
\caption{a $(2,3)$-regular hypergraph and its associated biregular bipartite graph where all vertices in $V_2$ have different neighborhoods in $V_1$}
\label{fig:bijection}
\end{figure}

\begin{definition}[adjacency matrix]
For  a hypergraph $H$ with $n$ vertices, we associate a $n\times n$ symmetric  matrix $A$ called the \textit{adjacency matrix} of $H$. For $i\not=j$, we define $A_{ij}$ as the number of hyperedges containing both $i$ and $j$; we define $A_{ii}=0$ for all $1\leq i\leq n$. When the hypergraph is $2$-uniform (i.e., it is a graph), this is the usual definition for the adjacency matrix of a graph.
\end{definition}
 	
The following lemma connects the adjacency matrix of a regular hypergraph with its associated biregular bipartite graph. It formally appears in \cite{sole1996spectra,dumitriu2019spectra}, and it is also informally mentioned in \cite{feng1996spectra}.

\begin{lemma}[Lemma 4.5 in \cite{dumitriu2019spectra}] \label{lem:matrixcorre}
Let $H$ be a $(d_1,d_2)$-regular hypergraph, and let $G$ be the corresponding $(d_1,d_2)$-biregular bipartite graph. Let $A_H$ be the adjacency matrix of $H$ and $A_G$ be the adjacency matrix of $G$ given by
\begin{align}\label{MatrixB}
A_G=\begin{pmatrix}
	0 & X\\
	X^{\top}& 0
\end{pmatrix}.
\end{align} Then
	$A_H=XX^{\top}-d_1I$.
\end{lemma}

\begin{definition}[walks and cycles]\label{cycledefinition}
A \textit{walk} of length $l$ on a hypergraph $H$ is a vertex-hyperedge sequence $(i_0,e_1,i_1,\cdots ,e_l, i_l)$ such that $i_{j-1}\not=i_{j}$ and $\{i_{j-1},i_j\}\subset e_j$ for all $1\leq j\leq l$. A walk is closed if $i_0=i_l$. A cycle of length $l$ in a hypergraph $H$ is a closed walk $(v_0,e_1,\dots, v_{l-1}, e_l,v_{l+1})$ such that 
all edges are distinct and all vertices are distinct subject to $v_{l+1}=v_0$. In the associated bipartite graph $G$, a cycle of length $2l$ corresponds to a cycle of length $l$ in $H$.
\end{definition}

Let $\mathcal G(n,m,d_1,d_2)$ be the set of all  simple biregular bipartite random graphs  with vertex set $V=V_1\cup V_2$ such that  $|V_1|=n, |V_2|=m$, and every vertex in $V_i$ has degree $d_i$ for $i=1,2$.  Without loss of generality, we assume $d_1\geq d_2$. Let $\mathcal H (n,d_1,d_2)$ be the set of all simple (without multiple hyperedges) $(d_1,d_2)$-regular hypergraphs with labeled  {vertex set $[n]$ and $\frac{nd_1}{d_2}$ many labeled hyperedges denoted by $\{e_1,\dots,e_{nd_1/d_2}\}$.} 

\begin{remark}
We  can also consider all $(d_1,d_2)$-regular hypergraphs with labeled vertices and unlabeled hyperedges. Since all hyperedges are distinct, any such regular hypergraph with unlabeled hyperedges corresponds to $(nd_1/d_2)!$ regular hypergraphs with labeled hyperedges. 
\end{remark}

\begin{figure}[ht]
	\includegraphics[width=4cm]{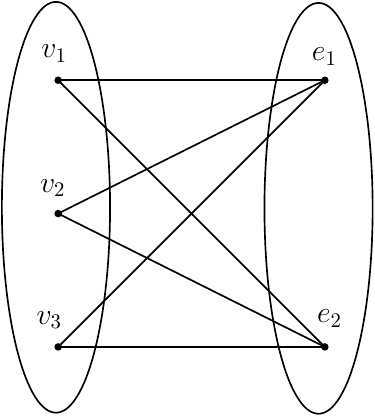}
	\caption{a subgraph in a  biregular bipartite graph which gives  multiple hyperedges $e_1$ and $e_2$ in the corresponding regular hypergraph}\label{fig:forbidden}
\end{figure}

It is well known (see for example \cite{feng1996spectra}) that the map $\Phi$ defined can be extended to a bijection $\tilde{\Phi}$ between labeled regular multi-hypergraphs and biregular bipartite graphs.  See  Figure \ref{fig:bijection} as an example of the bijection. For a given biregular bipartite graph, if there are two vertices in $V_2$ that have the same set of neighbors in $V_1$, the corresponding regular hypergraph will have  multiple hyperedges, see Figure \ref{fig:forbidden}.  Let $\mathcal G'(n,m,d_1,d_2)$ be a subset of $\mathcal G(n,m,d_1,d_2)$ such that for any $G\in  \mathcal G'(n,m,d_1,d_2)$, any two vertices in $V_2$ have different  neighborhoods in $V_1$. The following lemma holds.

\begin{lemma}[Lemma 4.2 in \cite{dumitriu2019spectra}]\label{bijection}
	$\Phi$ is the restriction of the bijection $\tilde{\Phi}$ to  $\mathcal H(n,d_1,d_2)$  and its image is $\mathcal G'\left(n, m,d_1,d_2\right)$. Hence $|\mathcal H(n,d_1,d_2)|=|\mathcal G'\left(n, m,d_1,d_2\right)|$.
\end{lemma}
 From Lemma \ref{bijection}, the uniform distribution on $\mathcal G' \left(n,m,d_1,d_2\right)$ for biregular bipartite graphs  induces the uniform distribution  on $\mathcal H (n, d_1,d_2)$ for regular hypergraphs. With this observation, we can translate some of the results for spectra of random biregular bipartite graphs into results for spectra of random regular hypergraphs. A similar approach was applied in \cite{blinovsky2016asymptotic} to enumerate uniform hypergraphs with given degrees.

\begin{lemma}[Lemma 4.8 in  \cite{dumitriu2019spectra}]\label{lem:RHRG}
Let $G$ be a random biregular bipartite graph sampled uniformly from $\mathcal G\left(n,m,d_1,d_2\right)$ such that $
3\leq d_2\leq d_1\leq \frac{n}{32}$.  Let $\mathcal G'\left(n,m,d_1,d_2\right)$ be the set of biregular bipartite graphs corresponding to simple regular hypergraphs.
Then   
\begin{align}\label{eq:low}
  \mathbb P\left(G\in \mathcal G'\left(n,m,d_1,d_2\right) \right)\geq 1-\left(\frac{nd_1}{d_2}\right)^2\left(\frac{4ed_2}{n}\right)^{d_2}.  
\end{align} 
In particular, 
\begin{align}\label{eq:low1}
  \mathbb P\left(G\in \mathcal G'\left(n,m,d_1,d_2\right) \right)=1-O\left( \frac{d_1^2}{nd_2^2}\right).  
\end{align}
\end{lemma}

Lemma \ref{lem:RHRG} implies the following total variation bound.
\begin{lemma}[total variation bound]\label{lem:TV}
Let $\mu_n$ be the probability measure of the random $(d_1,d_2)$-regular hypergraph with $n$ vertices induced on the set of all $(n,m,d_1,d_2)$-biregular bipartite graphs,  and let $\mu_n'$ be the uniform measure on the set of all $(n,m,d_1,d_2)$-biregular bipartite graphs. We have 
\begin{align}
    d_{\textnormal{TV}}(\mu_n,\mu_n')\leq \left(\frac{nd_1}{d_2}\right)^2\left(\frac{4ed_2}{n}\right)^{d_2}.
\end{align}
\end{lemma}
\begin{proof} Since $\mathcal G'(n,m,d_1,d_2)$ is the set of all biregular bipartite graphs that are bijective to regular hypergraphs. We have 
$\mu_n(\mathcal G'(n,m,d_1,d_2) )=1$ and $\mu_n'(\mathcal G'(n,m,d_1,d_2))=\frac{|\mathcal \mathcal G'(n,m,d_1,d_2)|}{|\mathcal G(n,m, d_1,d_2)|}$.  Let $\mathcal F$ be the power set of $\mathcal G(n,m,d_1,d_2)$.
Taking into account the fact that both $\mu_n$ and $\mu_n'$ are uniform measures, we obtain that
\begin{align*}
   d_{\textnormal{TV}}(\mu_n,\mu_n')&=\sup_{A\in \mathcal F}|\mu_n(A)-\mu_n'(A)|=|1-\mu_n'(\mathcal \mathcal G'(n,m,d_1,d_2)) |\\
   &=\mathbb P\left(G\not\in \mathcal G'\left(n,m,d_1,d_2\right) \right)\leq \left(\frac{nd_1}{d_2}\right)^2\left(\frac{4ed_2}{n}\right)^{d_2},
\end{align*}
where the last inequality is from  Lemma \ref{lem:RHRG}.
\end{proof}
Equipped with Lemma \ref{lem:TV}, we obtain several corollaries for random regular hypergraphs in the following subsections.

\subsection{Cycle counts}
Recall the definition of cycles in a hypergraph given in Definition \ref{cycledefinition}.
Let $C_k$ be the number of cycles of length $k$ in a $(d_1,d_2)$-regular hypergraph. The following result holds.
\begin{cor}
Let $H$ be a $(d_1,d_2)$-random regular  hypergraph with cycle counts $(C_k, k\geq 2)$. Let $(Z_k, k\geq 2)$ be independent Poisson random variables with $\displaystyle \mathbb EZ_k=\frac{(d_1-1)^{k}(d_2-1)^{k}}{2k}$.	 For any $n,m\geq 1$, $r\geq 3$, and  $3\leq d_2\leq d_1\leq  \frac{n}{32}$,
\[ d_{\textnormal{TV}}((C_2,\dots,C_r),(Z_2,\dots,Z_r))\leq  \frac{c_6\sqrt{r}(d_1-1)^{3r/2}(d_2-1)^{3r/2}}{nd_1}+\left(\frac{nd_1}{d_2}\right)^2\left(\frac{4ed_2}{n}\right)^{d_2}.\]
\end{cor}
\begin{proof}
Let $\tilde{C}_k$ be the number of cycles with length $2k$ in a uniform random $(d_1,d_2)$-biregular bipartite graph. From Lemma \ref{lem:TV},
\begin{align}
  d_{\textnormal{TV}}((C_2,\dots,C_r),(\tilde{C}_2,\dots,\tilde{C}_r))\leq d_{\textnormal{TV}}(\mu_n,\mu_n')\leq  \left(\frac{nd_1}{d_2}\right)^2\left(\frac{4ed_2}{n}\right)^{d_2}.  
\end{align}
Then the conclusion follows from Theorem \ref{thm:Poissoncount} and the triangle inequality.
\end{proof}

% In the same way we have total variation bound for the cyclically non-backtracking walks:

% \begin{lemma}
% Suppose $d_1\leq n^{1/3}$ and $r\leq n^{1/10}$. There exists a constant $c_8>0$ such that
% \begin{align}
% 	d_{\textnormal{TV}}\left( (\CNBW_k^{(n)},2\leq k\leq r),(\CNBW_{k}^{(\infty)},2\leq k\leq r)\right)\leq \frac{c_8\sqrt{r}[(d_1-1)(d_2-1)]^{3r/2}}{nd_1}+. 
% \end{align}
% \end{lemma}

\subsection{Global laws} The limiting spectral distributions for the adjacency matrix of a random regular hypergraph can be summarized in the following corollary.
\begin{cor}\label{thm:3regimesglobal}
Let $H$ be a random $(d_1,d_2)$-regular hypergraph. 
\begin{enumerate}
\item If $d_1,d_2$ are fixed,   the empirical spectral distribution of $\frac{A-(d_2-2)}{\sqrt{(d_1-1)(d_2-1)}}$ converges in probability to a measure $\mu$ with density function given by \begin{align}\label{LSD}
	f(x):=\frac{1+\frac{d_2-1}{q}}{(1+\frac{1}{q}-\frac{x}{\sqrt{q}})(1+\frac{(d_2-1)^2}{q}+\frac{(d_2-1)x}{\sqrt{q}})}\frac{1}{\pi}\sqrt{1-\frac{x^2}{4}} dx,
	\end{align}
	where $q=(d_1-1)(d_2-1)$.
	 \item For $d_1,d_2\to\infty$ with $\frac{d_1}{d_2}\to\alpha\geq 1$ and $d_1\leq \frac{n}{32}$,  the empirical spectral distribution of $ \frac{A-(d_2-2)}{\sqrt{(d_1-1)(d_2-1)}}$ converges in probability to  a measure supported on $[-2,2]$ with a density function given by	
	 \begin{align}\label{mua}
	g(x)=\frac{\alpha}{1+\alpha+\sqrt{\alpha}x}\frac{1}{\pi}\sqrt{1-\frac{x^2}{4}} .
	\end{align}
    \item If $d_1\to\infty, d_1=o(n^{\epsilon})$ for any $\epsilon>0$ and $\frac{d_1}{d_2}\to\infty$,  the ESD of $\frac{A}{\sqrt{(d_1-1)(d_2-1)}}$ converges to the semicircle law in probability.
\end{enumerate}
\end{cor}

\begin{proof}[Proof of Corollary \ref{thm:3regimesglobal}]
Claim (1) is proved in Theorem 6.4 of \cite{dumitriu2019spectra} based on a result for deterministic regular hypergraphs in Theorem 5 of \cite{feng1996spectra}.

Claim (2) is a combination of several results. When $d_1=o(n^{1/2})$, it is proved in Theorem 6.6 of \cite{dumitriu2019spectra} based on the global law for random biregular bipartite graphs in \cite{dumitriu2016marvcenko} and \cite{tran2020local}.  When $d_1=\omega(\log^4 n)$, the optimal local law for RBBGs  was recently  proved in \cite{yang2017local}, which also implies the global law for RBBGs. 
When $d_1\leq \frac{n}{32}$ and $ \frac{d_1}{d_2}\to\alpha$, from Lemma \ref{lem:RHRG}, 
\[ \mathbb P\left(G\in \mathcal G'\left(n,m,d_1,d_2\right) \right)\to 1.\]
Therefore by the same proof of Theorem 6.6 in \cite{dumitriu2019spectra}, the ESD of $\frac{A-(d_2-2)}{\sqrt{(d_1-1)(d_2-1)}}$ for random regular hypergraphs converges in probability.

Under the assumptions $d_1\to\infty, d_1=o(n^{\epsilon})$ for any $\epsilon>0$ and $\frac{d_1}{d_2}\to\infty$, from  Lemma \ref{lem:RHRG}, we have again
$ \mathbb P\left(G\in \mathcal G'\left(n,m,d_1,d_2\right) \right)\to 1.$
Then Claim (3) follows from  Theorem \ref{thm:globallaw} and Lemma \ref{lem:TV}. 
\end{proof}

\begin{remark}
The ESD in Corollary \ref{thm:3regimesglobal} (2) is a shifted and scaled Mar\v{c}enko-Pastur law. Taking $\alpha\to \infty$, $g(x)$ converges to the density function of the semicircle law. The  transition from Mar\v{c}enko-Pastur law to the semicircle law was also proved for sample covariance matrices in \cite{bai1988convergence} when the aspect ratio goes to infinity.
\end{remark}
\begin{remark}
A semicircle law for the adjacency matrix of $d_2$-uniform Erd\H{o}s-R\'{e}nyi random hypergraphs with growing expected degrees was proved in Theorem 5 of \cite{lu2012loose} when $d_2$ is a constant. Part (3) of Corollary \ref{thm:3regimesglobal} proves a corresponding semicircle law for random $d_2$-uniform $d_1$-regular hypergraphs where $d_2$ can be a parameter depending on $n$.
\end{remark}

\subsection{Spectral gaps}
The spectral gap for random regular hypergraphs  with fixed $d_1,d_2$ was studied in \cite{dumitriu2019spectra}. Here we include the results for the case when  $d_1,d_2$ are growing with $n$.

\begin{cor}
Let	$H$ be a random $(d_1,d_2)$-regular hypergraph with $d_1\geq d_2$. Let $\lambda_1\geq \cdots\geq \lambda_n$ be the eigenvalues of $A$. Let $\lambda=\max_{2\leq i\leq n}|\lambda_i|$.
\begin{enumerate}
	\item Suppose $d_1\geq d_2\geq 3$ is fixed. There exists a sequence $\epsilon_n\to 0$ such that
	 \begin{align*}
	     \mathbb P( |\lambda-(d_2-2)|\geq 2\sqrt{(d_1-1)(d_2-1)}+\epsilon_n)\to 0 
	 \end{align*} 
	as $n\to\infty$.
	\item Suppose $3\leq d_2\leq \frac{1}{2}n^{2/3}$, $d_1\geq d_2\geq cd_1$ for some $c\in (0,1)$. Then for some constant $K>0$ depending on $c$, for all $n\geq 1$,
	\begin{align*}
	\mathbb P\left( \lambda\geq K\sqrt{(d_1-1)(d_2-1)}\right)=O\left(\frac{1}{n}\right).
	\end{align*} 
\item Suppose $3\leq d_2\leq C_1$ for a constant $C_1$, and $d_1=o(n^{1/2})$. There exists a constant $C$ depending on $C_1$ such that
	\begin{align*}
\mathbb P\left( \lambda\geq C\sqrt{(d_1-1)(d_2-1)}\right)=O\left(\frac{d_1^2}{n^2} \right). 	    
	\end{align*} 
\end{enumerate}
\end{cor}
\begin{proof}
Claim (1) is proved in Theorem 4.3 in \cite{dumitriu2019spectra}. Claim (2) and (3) follow from part (2) and (3) in Theorem \ref{thm:spectralgap} with Lemma \ref{lem:RHRG}.
\end{proof}

\begin{remark}
Results in \cite{zhu2020second} that Claim (2) and (3) are based on have stronger probability estimates. However, Lemma \ref{lem:RHRG} we used here yields a weaker failure probability. 
\end{remark}

\subsection{Eigenvalue fluctuations}

The following eigenvalue fluctuation results for random regular hypergraphs can be derived from  Lemma \ref{lem:matrixcorre}, Lemma \ref{lem:TV}, and the eigenvalue fluctuations results for random biregular bipartite graphs in Section \ref{sec:CLT}.

\begin{cor}\label{thm:hyperfixd}
For fixed $d_1\geq d_2\geq 3$, let $H$ be a random $(d_1,d_2)$-regular hypergraph with adjacency matrix $A$. Let $\lambda_1\geq \cdots\geq \lambda_n$ be the eigenvalues of $\frac{A}{\sqrt{(d_1-1)(d_2-1)}}$. Suppose $f$ is a function satisfying the same conditions in Theorem \ref{thm:CLTfixed}. Then  $Y_{f}^{(n)}:=\sum_{i=1}^n f(\lambda_i)-n a_0$ converges in distribution as $n\to\infty$ to the infinitely divisible random variable
\[ Y_{f}:=\sum_{k=2}^{\infty} \frac{a_k}{[(d_1-1)(d_2-1)]^{k/2}}\CNBW_k^{(\infty)},\]
where $\CNBW_k^{(\infty)}$ is defined in \eqref{eq:CNBWinfinity}.
\end{cor}

\begin{proof}
Let $\tilde{Y}_f^{(n)}$ be the corresponding random variable of $Y_{f}^{(n)}$ for the uniform random biregular bipartite graphs  considered in Theorem \ref{thm:CLTfixed}. From the total variation distance bound in Lemma \ref{lem:TV}, we have 
\begin{align*}
    d_{\textnormal{TV}}(Y_f^{(n)}, \tilde{Y}_{f}^{(n)})\leq  d_{\textnormal{TV}}(\mu_n,\mu_n')\leq \left(\frac{nd_1}{d_2}\right)^2\left(\frac{4ed_2}{n}\right)^{d_2}=o(1).
\end{align*}
Therefore $\tilde{Y}_f^{(n)}$ and $Y_f^{(n)}$ converge in distribution to the same law.
\end{proof}

\begin{cor}
Let $H$ be a random $(d_1,d_2)$-regular hypergraph with $d_1d_2\to\infty$ as $n\to\infty$ and $d_1d_2=n^{o(1)}$. Let $\lambda_1\geq \cdots\geq \lambda_n$ be the eigenvalues of $\frac{A}{\sqrt{(d_1-1)(d_2-1)}}$. Let $f$ be a function satisfying \eqref{eq:expbound} and \eqref{eq:frn}.
Suppose one of the two assumptions holds:
\begin{enumerate}
    \item  there exists a constant $c\geq 1$ such that $1\leq \frac{d_1}{d_2}\leq c$,
    \item $3\leq d_2\leq c_1$ for a constant $c_1\geq 3$.
\end{enumerate} 
Then the random variable
\[Y_f^{(n)}=\sum_{i=1}^n f(\lambda_i)- m_f^{(n)}\]
converges in law to a Gaussian random variable with mean zero and variance $\sigma_f=\sum_{k=2}^{\infty}2ka_k^2$. Moreover, for any fixed $t$, consider the entire functions $g_1,\dots, g_t$   satisfying \eqref{eq:expbound} and \eqref{eq:frn}. The corresponding random vector 
$(Y_{g_1}^{(n)},\dots, Y_{g_t}^{(n)})$ converges in distribution to a centered Gaussian random vector $(Z_{g_1},\dots, Z_{g_t})$ with covariance
\begin{align*}
    \textnormal{Cov}(Z_{g_i},Z_{g_j})=2\sum_{k=2}^{\infty}ka_k(g_i) a_k(g_j)
\end{align*}
for $1\leq i,j\leq t$,
where $a_k(g_i), a_k(g_j)$ are the $k$-th coefficients in the expansion \eqref{eq:fphi} for $g_i,g_j$, respectively.
\end{cor}

\begin{proof}
Recall Lemma \ref{lem:TV} and our assumption $d_1d_2=n^{o(1)}$.
Under Case (1), we have $d_2\to\infty$ and  
\begin{align*}
  d_{\textnormal{TV}}(\mu_n,\mu_n')\leq \left(\frac{nd_1}{d_2}\right)^2\left(\frac{4ed_2}{n}\right)^{d_2}=O(n^2)(n^{(-1+o(1))d_2})=o(1).
\end{align*}
Under Case (2), we have 
\begin{align*}
  d_{\textnormal{TV}}(\mu_n,\mu_n')\leq \left(\frac{nd_1}{d_2}\right)^2\left(\frac{4ed_2}{n}\right)^{d_2}=O(n^2d_1^2)\left( \frac{4ec_1}{n}\right)^{3}=o(1).
\end{align*}

Then with Lemma \ref{lem:TV}, in both cases $Y_f^{(n)}$ converges in distribution to the same limiting random variable defined in Theorem \ref{thm:GaussianCLT}. The proof of the covariance part follows in the same way.
\end{proof}

\bibliographystyle{plain}
\bibliography{globalref.bib}
\end{document}